\documentclass[11pt]{article}

\usepackage{epsfig}
\usepackage{amssymb, latexsym}
\usepackage{amscd}
\usepackage[all]{xy}
\usepackage{epsf}
\usepackage[mathscr]{eucal}

\usepackage{amsmath,amsfonts,amscd,amssymb,amsthm}

\usepackage{appendix}

\long\def\comment#1\endcomment{}

\theoremstyle{plain}

\newtheorem{theorem}{{\sc Theorem}}[section]
\newtheorem{lemma}[theorem]{\sc Lemma}

\newtheorem{prop}[theorem]{\sc Proposition}
\newtheorem{coroll}[theorem]{\sc Corollary}

\theoremstyle{plain}

\newtheorem{defn}[theorem]{\sc Definition}

\theoremstyle{exercise}
\newtheorem{remark}[theorem]{\sc Remark}

\makeatletter \@addtoreset{equation}{section} \makeatother

\def\eqref#1{\thetag{\ref{#1}}}

\let\latexref=\ref
\def\ref#1{{\normalfont{\latexref{#1}}}}

\newcommand{\ldot}{{\:\raisebox{2,3pt}{\text{\circle*{1.5}}}}}
\def\dlim_#1{{\displaystyle\lim_{#1}}^\hdot}

\newcommand{\id}{\operatorname{\rm id}}

\newcommand{\Ob}{\mathrm{Ob}}

\newcommand{\opp}{\mathrm{opp}}

\newcommand{\Hom}{\mathrm{Hom}}

\newcommand{\Cat}{{\mathscr{C}at}}
\newcommand{\Top}{{\mathscr{T}op}}
\renewcommand{\top}{\mathrm{top}}

\newcommand{\SSets}{{\mathscr{SS}ets}}

\newcommand{\str}{\mathrm{inj}}

\title{\sc{On Evrard's homotopy fibrant replacement \\of a functor}}

\author{\sc{Boris Shoikhet}}
\date{}

\begin{document}\maketitle
{\footnotesize
\begin{center}{\parbox{4,5in}{{\sc Abstract.}
We provide a more economical refined version of Evrard's categorical cocylinder factorization of a functor [Ev1,2]. We show that any functor between small categories can be factored into a homotopy equivalence followed by a (co)fibred functor which satisfies 
the (dual) assumption of Quillen's Theorem B.
}}
\end{center}
}

\section*{\sc Introduction}
The problem treated here is how to replace a functor $f\colon \mathscr{C}\to\mathscr{D}$ between small categories with a ``homotopy fibration of categories''. By latter we mean a functor fulfilling the assumption of Corollary to Quillen Theorem B [Q, Section 1], see Proposition \ref{propm} for several equivalent formulations.

Such a replacement of a functor $f\colon\mathscr{C}\to\mathscr{D}$ between small categories by a homotopy fibration is a commutative diagram
$$
\xymatrix{\mathscr{C}\ar[r]^{i\ \ }\ar[rd]_{f}&\mathscr{H}(f)\ar[d]^{f_h}\\
&\mathscr{D}
}
$$
where $i$ is a weak equivalence of categories, and $f_h$ is a homotopy fibration.

There are apparently many different explicit constructions for that. We consider here a construction of Marcel Evrard [Ev1,Ev2], which mimics the cocylinder factorization of a continuous map of topological spaces. His construction is based on a specific categorical model $\Lambda\mathscr{D}$ of the topological cocylinder $Y^{[0,1]}$. We shall review Evrard's construction and study in parallel a more economical version thereof. Our proof of its homotopical properties is completely new.
We use on one side Rezk's notion of a {\it sharp map} of simplicial sets [Re], which we rename {\it simplicial $h$-fibration}, and on the other side Grothendieck-Maltsiniotis' theory of {\it aspherical functors}, cf. [Ma].

Then the category $\mathscr{H}(f)$, for $f\colon\mathscr{C}\to\mathscr{D}$, is defined as the following pullback diagram
$$
\xymatrix{
\mathscr{H}(f)\ar[r]\ar[d]&\Lambda\mathscr{D}\ar[d]^{p_0}\\
\mathscr{C}\ar[r]^{f}&\mathscr{D}
}
$$
where the functor $p_1$ gives the functor $f_h\colon \mathscr{H}(f)\to\mathscr{D}$, the ``fibrant replacement'' of $f$, and $\Lambda(\mathscr{D})$ is the categorical analogue of the topological path space, see Section \ref{sectionpath} for its definition. Here $p_0, p_1\colon \Lambda(\mathscr{D})\to\mathscr{D}$ are the functors of the left and right ends of a path.

Recall the classical topological construction mentioned above. Let $f\colon X\to Y$ be a continuous map.
Consider the free path space $Y^{[0,1]}$, and the two projections $p_0,p_1: Y^{[0,1]}\to Y$, assigning to a path its left and right ends. Define $X_h$ as the following pullback diagram
$$
\xymatrix{
X_h\ar[r]\ar[d]&Y^{[0,1]}\ar[d]^{p_0}\\
X\ar[r]^{f}&Y
}
$$
where the projection $p_1$ defines a map $f_h\colon X_h\to Y$ known to be a Serre fibration.

The paper contains three Sections.

Section 1 contains categorical preliminaries, and the material here is fairly standard. The exception is Section \ref{section1.4} providing a characterization of the nerve of a homotopy fibrant functor via  Rezk's sharp maps. 
The more classical topics recalled here include Quillen Theorem B and the Corollary to it, and the categorical Grothendieck construction. 

In Section 2 we recall Evrard's constructions of the free path category $\Lambda(\mathscr{D})$ and of the functor $f_h\colon\mathscr{H}(f)\to\mathscr{D}$, and discuss different weak homotopy equivalence relations on functors between small categories. 
We prove that $\Lambda(\mathscr{D})$ is weakly homotopy equivalent to $\mathscr{D}$, and that the corresponding map $\mathscr{C}\to\mathscr{H}(f)$ is a week homotopy equivalent. Then we formulate the main result Theorem \ref{theoreme}.

Section 3 is devoted to a proof of the remaining part of Theorem \ref{theoreme}, that $f_h\colon\mathscr{H}(f)\to\mathscr{D}$ is a homotopy fibration. We study simultaneously two versions: Evrard's original construction as well as our refined version thereof.

\subsection*{}
\subsubsection*{\sc Acknowledgements}
The author is thankful to Bernhard Keller and Georges Maltsiniotis for sending to him, several years ago, a scanned copy of the M.Evrard's Thesis [Ev1].

The author is wholeheartedly thankful to the anonymous referee, whose detailed report on the submitted version of the paper and whose many suggestions made it possible to strengthen the main result of the paper and to correct many inaccuracies. 
Among the referee's suggestions, we would mention Proposition \ref{propm} as well as its proof, and his remark that when having replaced cospans by spans in the construction of $\mathscr{H}(f)$, the map $f_h$ becomes a homotopy fibration. 
Another important suggestion concerned a more careful treatment of three different concepts of homotopy relations on functors in Section \ref{diffhom}, which helped to make the proofs in Section 2 fairly rigorous.

The work was partially supported by the FWO research grant ``Kredieten aan navorsers'' project nr. 6525.

\section{\sc Categorical preliminaries}\label{section1}
In Sections \ref{section1.1}-\ref{section1.4} we recall basic facts on homotopy theory of small categories, including Quillen's Theorem B and its interpretation in terms of Rezk's sharp maps of simplicial sets, followed by the Grothendieck construction in Section \ref{gc}.
\subsection{\sc The basic principles of the homotopy theory of categories}\label{section1.1}
Here we recall some elementary facts about the homotopy theory of categories, following the first few pages of [Q2, Section 1].

The theory starts with the {\it classifying space} functor $B\colon \Cat\to \Top$ from small categories to topological spaces, introduced by G.Segal in [Seg1].

Firstly one defines the {\it nerve} $N\mathscr{C}$ of a small category $\mathscr{C}$, which is the simplicial set whose $n$-simplices are chains of $n$ composable morphisms:
\begin{equation}\label{eqn1.1}
X_0\xrightarrow{f_1}X_1\xrightarrow{f_2}X_2\cdots X_{n-1}\xrightarrow{f_n}X_n
\end{equation}
The $i$-th face map $\delta_i\colon N\mathscr{C}_n\to N\mathscr{C}_{n-1}$ is obtained by deleting of $X_i$ in \eqref{eqn1.1}, and if $i\ne 0,n$, by replacing the maps $f_{i}$ and $f_{i+1}$ by their composition.
The $i$-th degeneracy map $\varepsilon_i\colon N\mathscr{C}_n\to N\mathscr{C}_{n+1}$ is obtained by inserting of another copy of $X_i$ at the $i$-th position, and inserting the identity map between the two copies $X_i$.

It is a simplicial set, functorially depending on $\mathscr{C}$. The geometric realization of $N\mathscr{C}$ is a topological space, called {\it the classifying space} of $\mathscr{C}$. It is denoted by $B\mathscr{C}$:
\begin{equation}\label{eqn1.2}
B\mathscr{C}=|N\mathscr{C}|
\end{equation}

Any functor $f\colon\mathscr{C}_1\to\mathscr{C}_2$ defines a map of topological spaces $B(f)\colon B\mathscr{C}_1\to B\mathscr{C}_2$.
For three categories $\mathscr{C}_1,\mathscr{C}_2,\mathscr{C}_3$ and functors $f\colon \mathscr{C}_1\to\mathscr{C}_2, g\colon \mathscr{C}_2\to\mathscr{C}_3$ one has:
\begin{equation}\label{eqn1.2bis}
B(g\circ f)=B(g)\circ B(f)
\end{equation}
as the nerve enjoys this property, and the geometrical realization is a functor.

\begin{defn}\label{defhq}{\rm
\begin{itemize}
\item[(i)]
Two functors $f,g\colon \mathscr{C}_1\to\mathscr{C}_2$ are said to be {\it homotopic} (in the sense of Quillen) if the corresponding maps
$Bf,Bg\colon B\mathscr{C}_1\to B\mathscr{C}_2$ are homotopic maps of topological spaces;
\item[(ii)] A functor $f\colon \mathscr{C}_1\to \mathscr{C}_2$ is called a {\it weak homotopy equivalence} if the corresponding map $Bf\colon B\mathscr{C}_1\to B\mathscr{C}_2$ is a homotopy equivalence of topological spaces;
\item[(iii)] A functor $f\colon \mathscr{C}_1\to\mathscr{C}_2$ is called a {\it homotopy equivalence} if there is a functor $g\colon \mathscr{C}_2\to\mathscr{C}_1$ such that $Bf\colon B\mathscr{C}_1\to B\mathscr{C}_2$ and $Bg\colon B\mathscr{C}_2\to B\mathscr{C}_1$ are homotopy inverse maps of topological spaces;
\item[(iv)] Two small categories $\mathscr{C}_1$ and $\mathscr{C}_2$ are called {\it weakly homotopy equivalent} (respectively, {\it homotopy equivalent}) if there is a functor $f\colon \mathscr{C}_1\to\mathscr{C}_2$ which is a weak homotopy equivalence (correspondingly, a homotopy equivalence).
\end{itemize}
}
\end{defn}

The data consisting of two functors $f,g\colon \mathscr{C}_1\to\mathscr{C}_2$ and a natural transformation $h\colon f\to g$ can be interpreted in the following way. Denote by $\mathscr{I}$ the category with two objects $0$ and $1$ and the only non-identity morphism $i\colon 0\to 1$.
Then the above data is the same that a single functor $F_{f,g,h}\colon \mathscr{C}_1\times \mathscr{I}\to \mathscr{C}_2$.
It results in a map:
\begin{equation}
B\mathscr{C}_1\times I\sim B\mathscr{C}_1\times B\mathscr{I}\to B(\mathscr{C}_1\times\mathscr{I})\to B\mathscr{C}_2
\end{equation}
where $I$ is the closed interval, $I=B\mathscr{I}$.

We have:

\begin{prop}\label{propbasic}
\begin{itemize}
\item[(i)]
Let $\mathscr{C}_1,\mathscr{C}_2$ be two small categories, $f,g\colon\mathscr{C}_1\to\mathscr{C}_2$ two functors, and $h\colon f\to g$ a map of functors. Then $h$ defines a homotopy between $B(f),B(g)\colon B\mathscr{C}_1\to B\mathscr{C}_2$;
\item[(ii)]
let $\mathscr{C}$ and $\mathscr{D}$ be small categories, and $$F\colon \mathscr{C}\rightleftarrows\mathscr{D}: G$$ a pair of
adjoint functors. Then $B\mathscr{C}$ and $B\mathscr{D}$ are homotopy equivalent topological spaces;
\item[(iii)]
suppose a small category $\mathscr{C}$ has an initial (resp., a final object). Then $B\mathscr{C}$ is contractible.
\end{itemize}
\end{prop}
\begin{proof}
We have just shown (i). The claim (ii) follows immediately from (i) and from \eqref{eqn1.2bis}, as there are adjunction maps of functors $F\circ G\to \id$ and $G\circ F\to \id$. For (iii), if $\mathscr{C}$ has an initial (resp., a final) object, the projection functor of $\mathscr{C}$ to the category $*$ with a single object and with the only (identity) morphism, admits a left (resp. a right) adjoint.
\end{proof}

For further references, recall here the following
\begin{prop}\label{geomreal}
The geometric realization functor $|\ |\colon \SSets\to\Top$ commutes with arbitrary colimits and with finite limits.
\end{prop}
\begin{proof}
For the first statement, the geometric realization is left adjoint to the singular complex functor, see e.g. [GJ, Prop.2.2]. For the second statement, see [May, Theorem 14.3].
\end{proof}

\subsection{\sc Quasi-fibrations and $h$-fibrations}\label{qfibr}
Here we recall a weaker version of the concepts of a Serre fibration in the category of topological spaces, and of a Kan fibration in the category of simplicial sets. This material will be used later in Section \ref{section1.4} to give a homotopy theoretical description of the map $Nf\colon N\mathscr{C}_1\to N\mathscr{C}_2$, where a functor $f\colon \mathscr{C}_1\to \mathscr{C}_2$ satisfies the assumption of Quillen Theorem B (correspondingly, the assumption of the Corollary to Quillen Theorem B), see Proposition \ref{propm} below. 

The following definition (for the topological case) was introduced in [DT].

Let $f\colon X\to Y$ be a map of topological spaces, $y\in Y$. The fiber $f^{-1}(y)$ is defined by the following pullback diagram
$$
\xymatrix{
f^{-1}(y)\ar[r]\ar[d]&X\ar[d]^{f}\\
\{y\}\ar[r]&Y
}
$$
The homotopy fiber $f^{-1}_h(y)$ is the homotopy limit of the same diagram. There is a natural map of topological spaces
$$
i_y\colon f^{-1}(y)\to f^{-1}_h(y)
$$

\begin{defn}\label{defqf}
{\rm
A map of topological spaces $f\colon X\to Y$ is called a {\it quasi-fibration} if for any point $y\in Y$ the map $i_y$ is a weak homotopy equivalence.
}
\end{defn}
It was shown in [DT] that a topological quasi-fibration gives rise to a long exact sequence in homotopy groups, in the same way as a Serre fibration does. 

It is known that the pullback of a topological quasi-fibration $f\colon X\to Y$ along a map $Y^\prime\to Y$ is {\it not}, in general, a quasi-fibration. In the same time, quasi-fibrations ``patch'' in the following sense: if for a map of topological spaces $f\colon X\to Y$ each point $y\in Y$ has a neighbourhood $U_y$ such that the restriction of $f$ to $U_y$ is a quasi-fibration, then $f$ also is a quasi-fibration. The reader is referred to [DT] and to the letter of Goodwillie [G] for further detail and proofs. 

For any closed model category $\mathscr{M}$ [Q1,Hi], Rezk [R] defines a class of {\it sharp maps}. In the category of topological spaces with Quillen's model structure, this class is contained in the class of {\it quasi-fibrations} of Dold-Thom, but has the advantage over the latter of being closed under base-change.

Let us recall the definition.
\begin{defn}\label{sharpdef}{\rm
A map $f\colon X\to Y$ in a closed model category $\mathscr{M}$ is called {\it sharp} if for each diagram
$$
\xymatrix{
A\ar[r]^i\ar[d]&A^\prime\ar[r]\ar[d]&X\ar[d]^f\\
B\ar[r]^j&B^\prime\ar[r]&Y
}
$$
in which each square is a pullback square, $j$ is a weak equivalence, the map $i$ is also a weak equivalence.  
}
\end{defn}
It follows immediately that the class of sharp maps is closed under base change. The closed model category is right proper [Hi, Sect. 13.1] if and only if each fibration is sharp [R, Prop.2.2].

In this paper, we call the sharp maps of Rezk {\it homotopy fibrations}, or {\it $h$-fibrations}.

Both categories $\Top$ of topological spaces and $\SSets$ of simplicial sets are right (and left) proper [Hi, Theorems 13.1.11, 13.1.13]. Therefore, in both categories the fibrations are $h$-fibrations, and thus the $h$-fibrations are considered as ``homotopy fibrations''. 

For the category $\Top$ any $h$-fibration is clearly a quasi-fibration. Thus, $h$-fibrations provide a more rigid concept than quasi-fibrations, and strictly more rigid, as the class of quasi-fibrations is not closed under base-change. It is likely that the topological $h$-fibrations are precisely those quasi-fibrations which remain quasi-fibrations under base-change, i.e. those which Goodwillie [G] calls {\it universal} quasi-fibrations.

For the category $\Top$ there is not known any checkable criterion for a map to an $h$-fibration. 

For the case of category of simplicial sets $\SSets$
Rezk proves the following result [R, Theorem 4.1 and Remark 4.2]:
\begin{prop}\label{proprezk}
The following statements are equivalent in $\SSets$:
\begin{itemize}
\item[(1)] $f\colon X\to Y$ is an $h$-fibration,
\item[(2)] for each map $g\colon \Delta[n]\to Y$ the pullback square 
$$
\xymatrix{
\Delta[n]\times_Y X\ar[r]\ar[d]&X\ar[d]^{f}\\
\Delta[n]\ar[r]^{g}&Y
}
$$
is homotopy cartesian,
\item[(3)] for each diagram of pullback squares of the form
$$
\xymatrix{
P\ar[r]^{h}\ar[d]&P^\prime\ar[r]\ar[d]&X\ar[d]^{f}\\
\Delta[m]\ar[r]^{\delta}&\Delta[n]\ar[r]&Y
}
$$
and arbitrary map $\delta$ of the standard simplices, the map $h$ is a weak equivalence.
\end{itemize}
\end{prop}
Note that the condition (2) is reminiscent of the definition of quasi-fibration of topological spaces.

In fact, one has
\begin{prop}\label{boryaqf}
Let $f\colon X\to Y$ be an $h$-fibration in $\SSets$. Then the geometric realization $|f|\colon |X|\to |Y|$ is a topological quasi-fibration.
\end{prop}
\begin{proof}
We need to prove that for any point $y\in |Y|$, the diagram 
\begin{equation}\label{bqf}
\xymatrix{
|f|^{-1}(y)\ar[r]\ar[d]&|X|\ar[d]^{|f|}\\
\mathrm{pt}\ar[r]^{y}&|Y|
}
\end{equation}
is homotopy cartesian. 

Assume at first that $y$ is a ``simplicial'' point in $|Y|$, that is, it comes from a 0-simplex in $Y$. Then the fact that \eqref{bqf} is homotopy cartesian follows from Proposition \ref{proprezk}(2) for $n=0$, and from Proposition \ref{geomreal}.
For general, not necessarily simplicial, point $y\in |Y|$, we find a simplicial point $y_0\in |Y|$ and a path connecting $y$ with $y_0$ (such $y_0$ always exists). This path defines a homotopy between $|f|^{-1}(y_0)$ and $|f|^{-1}(y)$, and the homotopy limit does not change when a map in the diagram is replaced by a homotopic one.
\end{proof}

Thus, we have the following implications:
\begin{equation}
\left(\mbox{simplicial $h$-fibrations}\right)\Rightarrow\left(\mbox{topological quasi-fibrations}\right)\Leftarrow
\left(\mbox{topological $h$-fibrations}\right)
\end{equation}
In the same time, the geometrical realization of a simplicial $h$-fibration may be {\it not} a topological $h$-fibration (at least, we do not see any argument which proves the contrary).

\subsection{\sc Pre(co-)fibred categories, and Quillen Theorem B}\label{sectionq}
Here we recall the definitions of {\it (pre-)fibred} and of {\it (pre-)cofibred} categories (due to Grothendieck [SGA1, Expos\'{e} VI]), and formulate Quillen's Theorem B and his Corollary to Theorem B [Q2, Section 1]. In the next Subsection, we discuss the geometrical counter-part of the Corollary to Quillen Theorem B, and link it with the Rezk's sharp maps in $\SSets$ in Proposition \ref{propm}.

Let $f\colon\mathscr{C}\to\mathscr{D}$ be a functor, $Y$ a fixed object of $\mathscr{D}$.

Denote by $Y\setminus f$ the category whose objects are pairs $(X,v)$ where $X\in\mathscr{C}$, and $v\colon Y\to f(X)$ is a morphism in $\mathscr{D}$. A morphism $(X,v)\to (X^\prime, v^\prime)$ is a morphism $w\colon X\to X^\prime$ in $\mathscr{C}$ such that $f(w)\circ v=v^\prime$.

As well, denote by $f\setminus Y$ the category whose objects are pairs $(X,v)$ where $X\in\mathscr{C}$, and $v\colon f(X)\to Y$ is a morphism in $\mathscr{D}$. A morphism $(X,v)\to (X^\prime,v^\prime)$ is a morphism $w\colon X\to X^\prime$ in $\mathscr{C}$ such that $v^\prime\circ f(w)=v$.

The Quillen's Theorem A [Q2, Section 1] says that if, for a functor $f\colon\mathscr{C}\to\mathscr{D}$ the category $Y\setminus f$ (correspondingly, $f\setminus Y$) is contractible for each $Y\in\mathscr{D}$ then $f$ is a weak homotopy equivalence (that is, $f_\top: B\mathscr{C}\to B\mathscr{D}$ is a homotopy equivalence).

Along with the comma categories $Y\setminus f$ and $f\setminus Y$, one considers the ``set-theoretical fiber'' $f^{-1}Y$. It is the subcategory of $\mathscr{C}$ of objects $X$ such that $f(X)=Y$ and of morphisms $w\colon X\to X^\prime$ such that $f(w)=\id_Y$.

The comma categories $Y\setminus f$ and $f\setminus Y$ are advantageous, comparably with $f^{-1}(Y)$, by their functorial behaviour.
Let $v\colon Y\to Y^\prime$ be a morphism in $\mathscr{D}$. Then one has the natural functors
$$
[v^*]\colon Y^\prime\setminus f\to Y\setminus f\text{  and  }[v_*]\colon f\setminus Y\to f\setminus Y^\prime
$$
However, for existence of base-change functors $v^*\colon f^{-1}(Y^\prime)\to f^{-1}(Y)$ and $v_*\colon f^{-1}(Y)\to f^{-1}(Y^\prime)$ one should assume that $f$ is pre-(co)fibred, see below.

We use the notations $Y\setminus \mathscr{D}$ and $\mathscr{D}\setminus Y$, for $Y\in \mathscr{D}$. The category $Y\setminus \mathscr{D}$ has as its objects the pairs $(Y^\prime,v)$ where $v\colon Y\to Y^\prime$, and a morphism
$t\colon (Y^\prime,v)\to (Y^{\prime\prime},u)$ is a morphism $t\colon Y^\prime\to Y^{\prime\prime}$ such that $t\circ v=u$, and the category $\mathscr{D}\setminus Y$ is defined accordingly.

The category $Y\setminus\mathscr{D}$ has $Y\xrightarrow{\id}Y$ as its initial object, and the category $\mathscr{D}\setminus Y$ has $Y\xrightarrow{\id}Y$ as its final object. Therefore, both categories $Y\setminus\mathscr{D}$ and $\mathscr{D}\setminus Y$ are contractible, by Proposition \ref{propbasic}(iii).

\begin{theorem}[Quillen Theorem B]\label{qtheoremb}
Let $f\colon\mathscr{C}\to\mathscr{D}$ be a functor such that for any arrow $v\colon Y\to Y^\prime$ in $\mathscr{D}$ the induced functor $[v^*]\colon Y^\prime\setminus f\to Y\setminus f$ is a weak homotopy equivalence. Then for any $Y\in\mathscr{D}$ the cartesian square of categories
\begin{equation}
\xymatrix{
Y\setminus f\ar[r]^j\ar[d]_{[f]}&\mathscr{C}\ar[d]^f\\
Y\setminus \mathscr{D}\ar[r]^{[j]}&\mathscr{D}
}
\end{equation}
is homotopy cartesian,
where
\begin{equation}
j(X,v)=X,\ \ [f](X,v)=(f(X),v),\ \ [j](Y^\prime,v)=Y^\prime
\end{equation}
Dually, assume that for any $v\colon Y\to Y^\prime$ the induced functor $[v_*]$ is a weak homotopy equivalence. Then the cartesian diagram
\begin{equation}
\xymatrix{
f\setminus Y\ar[r]^{j^\prime}\ar[d]_{(f)}&\mathscr{C}\ar[d]^{f}\\
\mathscr{D}\setminus Y\ar[r]^{[j^\prime]}&\mathscr{D}
}
\end{equation}
is homotopy cartesian.
\end{theorem}
\qed

There are the natural imbeddings functor
\begin{equation}
i_Y\colon f^{-1}Y\to Y\setminus f,\ \ X\mapsto (X,\id_Y)
\end{equation}
and
\begin{equation}
j_Y\colon f^{-1}Y\to f\setminus Y,\ \ X\mapsto (X,\id_Y)
\end{equation}

Let $f\colon \mathscr{C}\to\mathscr{D}$ be a functor. The category $\mathscr{C}$ is called {\it pre-fibred} over $\mathscr{D}$ via $f$, if for any $Y\in \mathscr{D}$ the functor $i_Y$ has a right adjoint.

As well, the category $\mathscr{C}$ is called {\it pre-cofibred} over $\mathscr{D}$ via $f$, if for any $Y\in\mathscr{D}$ the functor $j_Y$ has a left adjoint.

The right adjoint functor $R_Y$ to $i_Y$ (if it exists) assigns to each $(X,v)$ an object of $f^{-1}Y$, which is denoted by $v^*X$.

The left adjoint functor $L_Y$ to $j_Y$ (if it exists) assigns to each $(X,v)$ an object of $f^{-1}Y$, denoted by $v_*X$.

Let $v\colon Y\to Y^\prime$ be a morphism in $\mathscr{D}$.
Then the composition
\begin{equation}\label{bc}
v^*:=  R_Y  \circ [v^*]\circ i_{Y^\prime}: f^{-1}Y^\prime\to f^{-1}Y
\end{equation}
is called the {\it base-change functor}.

As well, for the same $v\colon Y\to Y^\prime$,
the composition
\begin{equation}
v_*:=L_{Y^\prime}\circ [v_*]\circ j_{Y}\colon f^{-1}Y\to f^{-1}Y^{\prime}
\end{equation}
is called the {\it cobase-change} functor.

The Corollary of Quillen's Theorem B says the following:
\begin{theorem}[Corollary to Quillen Theorem B]\label{qcorollb}
Let $f\colon \mathscr{C}\to\mathscr{D}$ be a functor. Assume that $\mathscr{C}$ is pre-fibred (corresp., pre-cofibred) over $\mathscr{D}$ via $f$, and that for any morphism $u\colon Y\to Y^\prime$ in $\mathscr{D}$
the base change functor $u^*\colon f^{-1}Y^\prime\to f^{-1}Y$ (corresp., the cobase change functor $u_*\colon f^{-1}Y\to f^{-1}Y^\prime$) is a weak homotopy equivalence. Then for any $Y\in\mathscr{D}$ the category $f^{-1}Y$ is the homotopy fiber of $f$ over $Y$. More precisely, the cartesian diagram
\begin{equation}
\xymatrix{
f^{-1}Y\ar[r]\ar[d]&\mathscr{C}\ar[d]^{f}\\
\{Y\}\ar[r]&\mathscr{D}
}
\end{equation}
is homotopy cartesian, for any object $Y\in\mathscr{D}$.
\end{theorem}
This property on the level of topological spaces just means that $Bf: B\mathscr{C}\to B\mathscr{D}$ is a quasi-fibration, with the fibers $B(f^{-1}Y)$ (see Definition \ref{defqf}).

It is clear that the fulfillment of the assumptions of Theorem \ref{qcorollb} implies the fulfillment of the assumptions of Theorem \ref{qtheoremb}. Indeed,
the functors $i_Y\colon f^{-1}Y\to Y\setminus f$ and $j_Y\colon f^{-1}Y\to f\setminus Y$ are homotopy equivalences, as they have adjoints, by Proposition \ref{propbasic}(ii). Then to say that $v^*$ (corresp., $v_*$) is a homotopy equivalence is the same that to say that $[v^*]$ (corresp., $[v_*]$) is a homotopy equivalence.

In fact, the assumption of Corollary to Quillen's Theorem B imply a stronger geometrical property than $Bf$ to be a quasi-fibration. Namely, one proves in Proposition \ref{propm} that $Nf$ is a simplicial homotopy fibration (a sharp map).

\subsection{\sc Quillen Theorem B and homotopy fibrations of simplicial sets}\label{section1.4}
Here we prove a result\footnote{The author is indebted to the anonymous referee for Proposition \ref{propm}, as well as for its proof.} providing a ``geometric interpretation'' of the Corollary to Quillen Theorem B.

In this section, we use the following ``categorification'' of the simplicial category $\Delta$.

Denote by $\Delta_n$ the {\it category} with $n+1$ objects generated by the linear graph
$$
0\rightarrow 1\rightarrow 2\rightarrow\dots\rightarrow n
$$
We can talk on functors $\Delta_m\to \Delta_n$. These functors are in 1-to-1 correspondence with the morphisms $[m]\to [n]$
in the simplicial category $\Delta$. Thus, we can recover the category $\Delta$ as the smallest full subcategory in $\Cat$, containing all categories $\Delta_n$, $n\ge 0$.

\subsubsection{\sc $h$-fibred and aspherical functors}\label{extrasect}
We say that a functor $f\colon \mathscr{C}\to\mathscr{D}$ is {\it $h$-fibred} if it is pre-fibred, and the base-change functor
$v^*\colon f^{-1}(Y^\prime)\to f^{-1}(Y)$ is a weak equivalence for any arrow $v\colon Y\to Y^\prime$ in $\mathscr{D}$.

One has
\begin{lemma}\label{base1}
The class of $h$-fibred functors is closed under base change.
\end{lemma}
\begin{proof}
One firstly shows that a functor $f\colon \mathscr{C}\to\mathscr{D}$ is a pre-fibred if and only if any morphism $u\colon a\to b$ in $\mathscr{D}$ has a {\it cartesian lift} whose end-point is any given object in $f^{-1}(b)$, that is, that the original Grothendieck definition [SGA1, Exp. VI], and the Quillen definition [Q2] of a pre-(co)fibred morphism agree, cf. [Ma, Lemme 1.1.16]. It implies that pre-fibred (corresp., fibred) functors are stable with respect to arbitrary base change, see [SGAI, Exp. VI, Cor. 6.9]. The fibers of the base-changed functor are identified with the fibers of $f$, as well as the corresponding morphisms $u^*$. 
\end{proof}

\begin{remark}{\rm
The pre-cofibred functors are also closed under arbitrary base-change, which is proven similarly. One can define an $h$-cofibred functor as a pre-cofibred functor $f\colon \mathscr{C}\to\mathscr{D}$ such that for any morphism $v\colon Y\to Y^\prime$ in $\mathscr{D}$ the corresponding functor $v_*\colon f^{-1}(Y)\to f^{-1}(Y^\prime)$ is a weak equivalence. Then one shows that the $h$-cofibred functors are stable under base-change.
}
\end{remark}

Recall the definition of {\it aspherical} functors from [Ma, Section 1.1]. A functor $f\colon \mathscr{C}\to\mathscr{D}$ is called {\it aspherical} (resp., {\it coaspherical}) if, for each object $d\in\mathscr{D}$, the 
category $\mathscr{C}/d$ (resp. $d/\mathscr{C}$) is contractible.
\begin{lemma}\label{aspherical}
\begin{itemize}
\item[(i)]
Any (co)aspherical functor is a weak equivalence;
\item[(ii)]
Let $f\colon \mathscr{C}\to\mathscr{D}$ be a functor, $\mathscr{C}$ has initial (resp., final) object $*_\mathscr{C}$, and $f$ takes the initial (resp., the final) object in $\mathscr{C}$ to an initial (resp., final) object $f(*_\mathscr{C})$ in $\mathscr{D}$. Then $f$ is aspherical (resp., coaspherical).
\end{itemize}
\end{lemma}
\begin{proof}
(i): It is Quillen Theorem A, see [Q2]. (ii): The category $\mathscr{C}/d$ (resp., $d/\mathscr{C}$) contains an initial (resp., final) object $f(*_\mathscr{C})\to d$ (resp., $d\to f(*_\mathscr{C})$), and the result follows from Proposition \ref{propbasic}(iii)
\end{proof}

One harder result we use in the proof of Proposition \ref{propm} below, is
\begin{lemma}\label{base2}
The aspherical (resp., coaspherical) functors are closed under base-change along pre-fibred (resp., pre-cofibred) functors.
\end{lemma}
This follows from [Ma, Corollary 3.2.13] and its dual, because any pre-fibred (resp., pre-cofibred) functor is ``smooth'' (resp., ``proper'') in Grothendieck-Maltsiniotis terminology, thanks to Proposition \ref{propbasic}(ii).

\qed

\subsubsection{\sc $h$-fibred functors and Quillen's Theorem B}
\begin{prop}\label{propm}
For a pre-fibred (resp., pre-cofibred) functor $f\colon \mathscr{C}\to\mathscr{D}$ the following conditions are equivalent:
\begin{itemize}
\item[(1)] $f$ is $h$-fibred (resp., $h$-cofibred),
\item[(2)] $f$ satisfies the assumption (resp., the dual assumption) of Quillen Theorem B,
\item[(3)] $Nf$ is a simplicial $h$-fibration, see Definition \ref{sharpdef}.
\end{itemize}
\end{prop}

\begin{proof}
We establish the Proposition for pre-fibred functors; the dual statement is proved in a dual manner. Observe that statement (3) is self-dual.

For $(1)\Leftrightarrow(2)$, consider for a morphism $v\colon Y\to Y^\prime$ in $\mathscr{D}$ the commutative diagram
\begin{equation}
\xymatrix{
f^{-1}(Y^\prime)\ar[r]^{v^*}\ar[d]_{i_{Y^\prime}}&f^{-1}(Y)\\
Y^\prime\setminus f\ar[r]^{[v^*]}&Y\setminus f\ar[u]_{R_{Y}}
}
\end{equation}
The vertical functors $i_{Y^\prime}$ and $R_Y$ admit adjoints, and thus are (weak) homotopy equivalences by Proposition \ref{propbasic}(ii). Then $v^*$ is a weak equivalence iff $[v^*]$ is.

For $(3)\Rightarrow (2)$, consider, for any objects $Y,Y^\prime\in\mathscr{D}$, and for a map $v\colon Y\to Y^\prime$, the diagram:
\begin{equation}
\xymatrix{
Y^\prime\setminus f\ar[r]^{[v_*]}\ar[d]&Y\setminus f\ar[r]\ar[d]&\mathscr{C}\ar[d]^{f}\\
Y^\prime\setminus \mathscr{D}\ar[r]^{(v_*)}&Y\setminus\mathscr{D}\ar[r]&\mathscr{D}
}
\end{equation}
whose both squares are Cartesian. The category $Y\setminus \mathscr{D}$ is the category whose objects are pairs $(X,Y\xrightarrow{s} X)$ where $X$ is an object of $\mathscr{D}$ and $s$ is a morphism in $\mathscr{D}$; the morphisms are defined in natural way. This category is contractible by Proposition \ref{propbasic}(iii), as $(Y,Y\xrightarrow{\id}Y)$ is its initial object.
Thus, the arrow $(v_*)$ is a weak equivalence. Assuming (3), we know that $[v_*]$ is a weak equivalence, if $(v_*)$ is, what gives (2).

The non-trivial part is $(1)\Rightarrow (3)$. The argument uses Lemma \ref{base2} of Maltsiniotis.

By Rezk's result Proposition \ref{proprezk}(3), it is enough to prove that in any diagram of simplicial sets
\begin{equation}\label{dzn}
\xymatrix{
A\ar[r]^{h}\ar[d]&B\ar[r]\ar[d]&N(\mathscr{C})\ar[d]^{N(f)}\\
\Delta[m]\ar[r]&\Delta[n]\ar[r]&N(\mathscr{D})
}
\end{equation}
whose squares are Cartesian, the map $h$ is a weak equivalence. By Yoneda lemma, 
$$
\Hom_{\SSets}(\Delta[n],X)=X_n
$$
Applying it to $X=N(\mathscr{D})$, we see that 
$$
\Hom_{\SSets}(\Delta[n],X)=\Hom_{\Cat}(\Delta_n,\mathscr{D})
$$
Thus, we can regard a diagram in $\Cat$, whose both squares are Cartesian:
\begin{equation}\label{dmn}
\xymatrix{
\mathscr{A}\ar[r]^{H}\ar[d]&\mathscr{B}\ar[r]\ar[d]&\mathscr{C}\ar[d]^{f}\\
\Delta_m\ar[r]^{\sigma}&\Delta_n\ar[r]&\mathscr{D}
}
\end{equation}
such that the diagram \eqref{dzn} it obtained by the term-wise application of the nerve functor to the diagram \eqref{dmn}.
It is thus enough to show that the functor $H$ is a weak equivalence of categories. 

The map $\sigma$ in \eqref{dmn} is a map $\sigma\colon [m]\to [n]$ in the category $\Delta$ (by slight abuse of notations, we denote both of these maps by $\sigma$).
We decompose the map $\sigma$ (on the categorical level) as the composition of two maps
$$
\Delta_m\xrightarrow{\sigma_1}\sigma(0)\setminus \Delta_n\xrightarrow{\sigma_2}0\setminus \Delta_n
$$
where the rightmost category $0\setminus \Delta_n$ is the same that $\Delta_n$. The first map $\sigma_1$ is naturally induced by $\sigma$, the second map $\sigma_2$ is $(v^*)$ for $v\colon 0\to \sigma(0)$.

Then our diagram \eqref{dmn} becomes one Cartesian square longer:
\begin{equation}\label{dmnbis}
\xymatrix{
\mathscr{A}_1\ar[r]^{H_1}\ar[d]&\mathscr{A}_2\ar[r]^{H_2}\ar[d]^{g_1}&\mathscr{B}\ar[r]\ar[d]^{g_2}&\mathscr{C}\ar[d]^{f}\\
\Delta_m\ar[r]^{\sigma_1}&\sigma(0)\setminus \Delta_n\ar[r]^{\sigma_2}&\Delta_n\ar[r]&\mathscr{D}
}
\end{equation}
Here $H=H_2\circ H_1$, and we prove that each of $H_1$ and $H_2$ is a weak equivalence.

We start with the easier case of $H_2$. The maps $g_1$ and $g_2$ in \eqref{dmnbis} are $h$-fibred, as $f$ is $h$-fibred by assumption, and by Lemma \ref{base1}. Both categories $\sigma(0)\setminus \Delta_n$ and $\Delta_n$ have initial objects.
The initial objects are $(\sigma(0),\id)$ and $0$, correspondingly. In the commutative diagram
\begin{equation}
\xymatrix{
\mathscr{A}_2\ar[rr]^{H_2}&&\mathscr{B}\\
g_1^{-1}(\sigma(0),\id)\ar[r]^{\id}\ar[u]^{i_1}&g_2^{-1}(\sigma(0))\ar[r]^{[v^*]}&g_2^{-1}(0)\ar[u]_{i_2}
}
\end{equation}
the embeddings $i_1$ and $i_2$ are weak equivalences, by Corollary to Quillen Theorem B. Moreover, $[v^*]$ is also a weak equivalence, as $f$ is $h$-fibred, and by (2). It implies that $H_2$ is a weak equivalence.

The proof that $H_1$ is also a weak equivalence relies on the theory of aspherical functors, see Section \ref{extrasect}. The functor $\sigma_1$ is aspherical by Lemma \ref{aspherical}, as it takes the initial object of $\Delta_m$ to the initial object of $\sigma(0)\setminus \Delta_n$.

As $g_1$ is $h$-fibred, Lemma \ref{base2} implies that $H_1$ is also aspherical. 
Then Lemma \ref{aspherical} implies that $H_1$ is a weak equivalence.

\end{proof}
\comment
(In the above proof, the functor $\sigma_2$ does not map the initial object to the initial one, and is not aspherical; therefore, the argument for $H_1$ is not applied to $H_2$. On the other hand, $\sigma_2$ is of the form $[v^*]$, when we identify $\Delta_n$ with $0\setminus \Delta_n$. As $f$ is $h$-fibred, it satisfies the assumption of Quillen Theorem B, by (2). It shows that $[v^*]$ is a weak equivalence. Given to that, $i_1$ and $i_2$ are weak equivalences, by Corollary to Quillen Theorem B. Note that Corollary to Quillen Theorem B is used only for functors whose target is a contractible category, and the result of (3) implies it in general.) 
\endcomment

\begin{remark}{\rm
Proposition \ref{propm} implies the Corollary to Quillen Theorem B, due to Proposition \ref{boryaqf}. However, the form given in 
Proposition \ref{propm}, is stronger than the original Quillen's result. 
}
\end{remark}

\subsection{\sc The fibred and cofibred Grothendieck constructions}\label{gc}
The Evrard construction discussed in Section 2 is expressed as the Grothendieck construction in category theory [SGA1, Expos\'{e}VI.8].
Here we recall the fibred and cofibred Grothendieck constructions, restricting ourselves to its properties necessary for the sequel. We refer the reader to loc.cit. and to [Th] for more detail.

Let $F\colon\mathscr{K}\to\Cat$ be a (strict) functor (an analogous construction also exists for $F$ a pseudo-functor).

There are two different constructions, both referred to as {\it the Grothendieck construction} of $F$, which are given as functors
$$
p_c\colon \mathscr{K}{\int _c}F\to \mathscr{K}\text{   and   }p_f\colon \mathscr{K}{\int_f}F\to\mathscr{K}^\opp
$$
where $p_c$ is cofibred and $p_f$ is fibred.

The objects of $\mathscr{K}\int_c F$ are pairs $(K,X)$ were $K\in \mathscr{K}$ and $X\in F(K)$.

A morphism $(k,x)\colon (K_1,X_1)\to (K_0,X_0)$ is given by a morphism $k\colon K_1\to K_0$ in $\mathscr{K}$ and a morphism
$x\colon F(k)(X_1)\to X_0$ in $F(K_0)$.

The composition is defined as $(k,x)\ldot (k^\prime, x^\prime)=(kk^\prime, x\ldot F(k)(x^\prime))$.

The map $p_c$ is defined as the projection onto the first component. The category $\mathscr{K}\int_c F$ is equivalent to the lax colimit of $F$, see Proposition \ref{colaxlimit}.

The objects of $\mathscr{K}\int_fF$ are pairs $(K,X)$ were $K\in \mathscr{K}$ and $X\in F(K)$.

A morphism $(k,x)\colon (K_1,X_1)\to (K_0,X_0)$ is given by a morphism $k\colon K_0\to K_1$ in $\mathscr{K}$ and a morphism
$x\colon X_1\to F(k)(X_0)$ in $F(K_1)$.

The composition is defined as $(k,x)\ldot (k^\prime,x^\prime)=(k^\prime k,F(k^\prime)(x)\ldot x^\prime)$.

The map $p_f$ is defined as the projection onto the first component. The category $\mathscr{K}\int_f F$ is equivalent to the lax limit of $F$, which is the statement dual to Proposition \ref{colaxlimit}.
\begin{lemma}\label{ntgr}
\begin{itemize}
\item[(1)] The functor $p_f\colon \mathscr{K}\int_f F\to\mathscr{K}^\opp$ 
expresses $\mathscr{K}\int_f F$ as a fibred category over $\mathscr{K}^\opp$,
the functor 
$p_c\colon \mathscr{K}\int_c F\to \mathscr{K}$ expresses $\mathscr{K}\int_c F$ as a cofibred category over $\mathscr{K}$;
\item[(2)] 
A natural transformation of functors $\theta\colon F\Rightarrow F^\prime\colon \mathscr{K}\to\Cat$ induces a functor
$$
[\theta]_f\colon \mathscr{K}\int_f F\to\mathscr{K}\int_f F^\prime
$$
of fibred categories over $\mathscr{K}^\opp$
and a functor
$$
[\theta]_c\colon \mathscr{K}\int_c F^\prime \to \mathscr{K}\int_c F
$$
of cofibred categories over $\mathscr{K}$.
\end{itemize}
\end{lemma}
\qed

Let $\theta_1,\theta_2\colon F\Rightarrow F^\prime\colon\mathscr{K}\to\Cat$ be two natural transformations of functors. In this context, it is possible to define a 3-arrow $\xi\colon \theta_1\Rrightarrow \theta_2$, conventionally called {\it a modification}.
Indeed, each $\theta_i$, $i=1,2$ is given by a functor $\theta_i(k)\colon F(k)\to F^\prime(k)$ such that for any $x\colon k_1\to k$ the corresponding diagram of functors commutes. Then a modification $\xi\colon \theta_1\Rrightarrow\theta_2$ is given, for each $k\in\mathscr{K}$, by a map $\xi(k)\colon \theta_1(k)\Rightarrow\theta_2(k)\colon F(k)\to F^\prime(k)$ of the corresponding functors, such that, for any $x\colon k_1\to k$, and for any $t\in F(k_1)$, the diagram
\begin{equation}
\xymatrix{
\theta_1(t)\ar[rr]^{\xi(k_1)(t)}\ar[d]_{F^\prime(x)}&&\theta_2(t)\ar[d]^{F^\prime(x)}\\
\theta_1(F(t))\ar[rr]^{\xi(k)(F(t))}&&\theta_2(F(t))
}
\end{equation}
commutes.

We can now extend Lemma \ref{ntgr} to the action of these modifications:
\begin{lemma}\label{ntgrbis}
Let $F,F^\prime\colon \mathscr{K}\to\Cat$ be two functors, $\theta_1,\theta_2\colon F\Rightarrow F^\prime$ maps of functors, and let $\xi\colon\theta_1\Rrightarrow\theta_2$ be a modification. Then $\xi$ induces a natural transformation of functors
$$
[\xi]_f\colon [\theta_1]_f\Rightarrow[\theta_2]_f\colon \mathscr{K}\int_f F\to\mathscr{K}\int_f F^\prime
$$
and a natural transformation of functors
$$
[\xi]_c\colon [\theta_2]_c\Rightarrow[\theta_1]_c\colon \mathscr{K}\int_c F^\prime\to\mathscr{K}\int_c F
$$
\end{lemma}

\qed

The following Proposition (see e.g. [Th, Prop. 1.3.1]) characterizes the Grothendieck construction category $\mathscr{K}\int_c F$ as the ``lax colimit'' of the functor $F\colon \mathscr{K}\to\Cat$.
\begin{prop}\label{colaxlimit}
Let $\mathscr{K}$ be a small category, $F\colon \mathscr{K}\to\Cat$ a strict functor, $\mathscr{C}$ a category. Then there is a bijection between the set of functors
$g\colon \mathscr{K}\int_c F\to\mathscr{C}$, and the set of data consisting of
\begin{itemize}
\item[(1)] for each object $K\in\mathscr{K}$, a functor $g(K)\colon F(K)\to\mathscr{C}$,
\item[(2)] for each morphism $k\colon K\to K^\prime$ in $\mathscr{K}$, a natural transformation
$$
g(k)\colon g(K)\to g(K^\prime)\circ F(k)
$$
\end{itemize}
such that $g(\id_K)=\id\colon g(K)\to g(K)$, and for $K^{\prime\prime}\xleftarrow{k^\prime} K^{\prime}\xleftarrow{k} K$ one has
\begin{equation}
g(k^\prime k)=g(k^\prime)\circ g(k)
\end{equation}
\end{prop}
See e.g. [Th, Prop. 1.3.1] for a short and direct proof.

\qed

There is a similar statement expressing that $\mathscr{K}\int_fF$ is a lax limit of $F$, it is left to the reader.

\section{\sc The Evrard construction and its refinement}
There is a nice construction due to M.Evrard [Ev1,2] of replacing a functor $f\colon\mathscr{C}\to\mathscr{D}$ between small categories by a functor $f_h\colon\mathscr{H}(f)\to \mathscr{D}$, where $\mathscr{H}(f)$ is another category, such that:
\begin{itemize}
\item[(i)] there is a functor $i\colon \mathscr{C}\to \mathscr{H}(f)$ which is a homotopy equivalence,
\item[(ii)] the nerve $N(f_h)\colon N(\mathscr{H}(f))\to N(\mathscr{D})$ is a simplicial $h$-fibration, see Section \ref{extrasect} and Proposition \ref{propm},
\item[(iii)] the diagram
\begin{equation}
\xymatrix{
\mathscr{C}\ar[r]^{i}\ar[dr]_{f}&\mathscr{H}(f)\ar[d]^{f_h}\\
&\mathscr{D}
}
\end{equation}
commutes.
\end{itemize}
Here we provide a more economical version of Evrard's construction enjoying the same properties (i)-(iii). Our proof applies simultaneously to Evrard's original construction and to our refinement.

\subsection{\sc A categorical cocylinder}\label{sectionpath}
We define a categorical cocylinder $\Lambda\mathscr{D}$ of a category $\mathscr{D}$ which is the analogue of the topological cylinder $Y^{[0,1]}$.

Define the category $\Lambda_n\mathscr{D}$ for $n\ge 0$. An object of $\Lambda_n\mathscr{D}$ is a zig-zag
\begin{equation}\label{notpath}
\bar{Y}_0\rightarrow Y_1\leftarrow \bar{Y}_1\rightarrow Y_2\leftarrow\bar{Y}_2\rightarrow Y_3\leftarrow \dots\rightarrow Y_n\leftarrow\bar{Y}_n
\end{equation}
which we denote by $\underline{Y}(n)$, and a morphism $\underline{Y}(n)\to \underline{Z}(n)$ is the set of maps $t_i\colon Y_i\to Z_i$ and
$\bar{t}_i\colon \bar{Y}_i\to\bar{Z}_i$ making all squares commutative. 
Such a morphism is denoted by $\underline{t}\colon \underline{Y}(n)\to\underline{Z}(n)$.
Note that $\Lambda_0\mathscr{D}=\mathscr{D}$.

We discuss simultaneously the original Evrard construction and its refinement, proposed here.
For that reason, we use two categories: $\Delta_\str$ and $\Delta_\le$. 

The category $\Delta_\str$ has objects are $\{[0],[1], [2],\dots\}$ and the morphisms $[m]\to [n]$ are the strictly order-preserving maps
from $0<\dots<m$ to $0<\dots<n$ (that is, such maps $\phi$ that $\phi(i)<\phi(j)$ for $i<j$; in particular a morphism exists only when $m\le n$).

The category $\Delta_\le$ has the same objects as $\Delta_\str$, and there is a single morphism $[m]\to [n]$ when $m\le n$, and the empty set of morphisms otherwise. For $m\le n$, we interpret the single morphism $[m]\to [n]$ as the map $\phi$ of ordered sets $\{0<1<\dots<m\}\to\{0<1<\dots<n\}$ defined as
$$
\phi(i)=n-m+i
$$

The assignment $[n]\mapsto\Lambda_n(\mathscr{D})$ gives rise to functors
$$
F_\str\colon \Delta_\str\to\Cat\text{  and  }F_\le\colon \Delta_\le\to\Cat
$$

In both cases, for an order-preserving morphism
$\phi\colon [m]\to [n]$, and for a $\underline{Y}\in \Lambda_m\mathscr{D}$, we set
\begin{equation}\label{lphi}
\begin{aligned}
\ &\Lambda(\phi)(\underline{Y})=\\
&\bar{Y}_0\xrightarrow{\id} \bar{Y}_0\xleftarrow{\id} \bar{Y}_0\xrightarrow{\id} \dots\xleftarrow{\id}\bar{Y}_0\xrightarrow{\id} {\bar{Y}_0}\leftarrow{Y}_1\rightarrow\underset{\substack{{\phi(1)\text{-th}}\\
\text{place}}}{\bar{Y}_1}\xleftarrow{\id}\bar{Y}_1\xrightarrow{\id} \dots\xrightarrow{\id}{\bar{Y}_1}\leftarrow Y_2\rightarrow\underset{\substack{{\phi(2)\text{-th}}\\ \text{place}}}{\bar{Y}_2}\xleftarrow{\id}\bar{Y}_2\xrightarrow{\id}\dots
\end{aligned}
\end{equation}
Finally, we define two categories $\Lambda_\str\mathscr{D}$ and $\Lambda_\le\mathscr{D}$ as the cofibrant Grothendieck constructions
\begin{equation}
\Lambda_\str\mathscr{D}=\Delta_\str\int_c F_\str
\end{equation}
and
\begin{equation}
\Lambda_\le\mathscr{D}=\Delta_\le\int_c F_\le
\end{equation}
We use the notation $\Lambda_*\mathscr{D}$ when the statement holds for both of these two categories.

Define two maps $p_0,p_1:\Lambda_*\mathscr{D}\to\mathscr{D}$ by setting
\begin{equation}
p_0([n],\underline{Y})=\bar{Y}_0,\ \ p_1([n],\underline{Y})=\bar{Y}_n
\end{equation}

\comment
Evrard gives [Ev1] the following definition:
\begin{defn}\label{defhe}{\rm
Let $f,g\colon \mathscr{C}_1\to\mathscr{C}_2$ be functors. They are called {\it homotopic in the sense of Evrard} if there is a functor
$H\colon \mathscr{C}_1\to\Lambda_n\mathscr{C}_2$, for some $n\ge 1$, such that $p_0\circ H=f$ and $p_1\circ H=g$. A functor $f\colon \mathscr{C}_1\to \mathscr{C}_2$ is called {\it a homotopy equivalence in the sense of Evrard} if there exists a functor $f^\prime\colon \mathscr{C}_2\to\mathscr{C}_1$ such that both compositions $f\circ f^\prime$ and $f^\prime\circ f$ are homotopy equivalent to the identity functors in the sense of Evrard.
}
\end{defn}
One easily shows that if there is a natural transformation $f\to g$ between two functors $f,g\colon \mathscr{C}_1\to\mathscr{C}_2$, they are homotopic in the sense of Evrard. That is, all statements of Proposition \ref{propbasic} hold if the homotopy is understood in the sense of Definition \ref{defhe}. The Evrard homotopy relation is a priori a weaker homotopy relation than the Quillen's one, see Definition \ref{defhq}.

We have:
\begin{prop}\label{prope1}
Let $f,g\colon \mathscr{C}_1\to\mathscr{C}_2$ be functors homotopic in the sense of Evrard. Then the corresponding maps $Bf,Bg\colon B\mathscr{C}_1\to B\mathscr{C}_2$ are homotopic. That is, the Evrard homotopy relation coincides with the Quillen homotopy relation.
\end{prop}
\begin{proof}
Let $f,g\colon\mathscr{C}_1\to\mathscr{C}_2$ be homotopic in the sense of Evrard. Then the corresponding functor $H\colon \mathscr{C}_1\to
\Lambda_n\mathscr{C}_2$ gives rise to $2n+1$ functors $$\bar{Y}_0,Y_1,\bar{Y}_1,Y_2,\dots,Y_n,\bar{Y}_n\colon \mathscr{C}_1\to\mathscr{C}_2$$
and natural transformations $\gamma_i\colon \bar{Y}_{i-1}\to Y_i$ and $\delta_i\colon Y_i\to\bar{Y}_i$, $i=1\dots n$. Moreover, $\bar{Y}_0=f$ and $\bar{Y}_n=g$. Now the result follows from Proposition \ref{propbasic}(i).
\end{proof}
In the same time, the Evrard homotopy relation leaves us more flexibility for constructions of homotopies, as we will now see.
\endcomment

Let $\underline{Y}\in\Lambda_n\mathscr{D}$ be a free path. Consider the path $\underline{p}_0(\underline{Y})$ which is, in the notations of \eqref{notpath}, the following path in $\Lambda_n\mathscr{D}$:
\begin{equation}
\bar{Y}_0\xrightarrow{\id}\bar{Y}_0\xleftarrow{\id}\bar{Y}_0\xrightarrow{\id}\bar{Y}_0\xleftarrow{\id}\bar{Y}_0\dots
\end{equation}
\begin{prop}\label{prope2}
The functor $\underline{p}_0\colon\Lambda_n\mathscr{D}\to\Lambda_n\mathscr{D}$ is connected to the identity functor of $\Lambda_n\mathscr{D}$ by a zig-zag of natural transformations of functors.
\end{prop}
\begin{proof}
For $k\le n$, define the objects $\underline{Y}^{(k)}$ and $\underline{Y}_{(k)}$ in $\Lambda_n\mathscr{D}$, as the paths
\begin{equation}
\bar{Y}_0\rightarrow Y_1\leftarrow \bar{Y}_1\rightarrow\dots\leftarrow\bar{Y}_{k-1}\rightarrow Y_k\xleftarrow{\id}Y_k\xrightarrow{\id}Y_k\leftarrow\dots
\end{equation}
and
\begin{equation}
\bar{Y}_0\rightarrow Y_1\leftarrow \bar{Y}_1\rightarrow\dots\leftarrow\bar{Y}_{k-1}\rightarrow Y_k\leftarrow \bar{Y}_k\xrightarrow{\id}\bar{Y}_k\xleftarrow{\id}\bar{Y}_k\rightarrow\dots
\end{equation}
correspondingly.

In these notations, $\underline{p}_0(\underline{Y})=\underline{Y}_{(0)}$, and $\underline{Y}=\underline{Y}_{(n)}$.

We claim that there is a zig-zag of natural transformations, from $\underline{p}_0$ to $\id$,
sending $\underline{Y}$ to the path
\begin{equation}
\underline{p}_0(\underline{Y})=\underline{Y}_{(0)}\xrightarrow{\alpha_1}\underline{Y}^{(1)}\xleftarrow{\beta_1}\underline{Y}_{(1)}\xrightarrow{\alpha_2}
\underline{Y}^{(2)}\xleftarrow{\beta_2}\underline{Y}_{(2)}\rightarrow\dots\xrightarrow{\alpha_n}
\underline{Y}^{(n)}\xleftarrow{\beta_n}\underline{Y}_{(n)}=\underline{Y}
\end{equation}
where the maps $\underline{Y}_{(i-1)}\xrightarrow{\alpha_{i}}\underline{Y}^{(i)}$ and $\underline{Y}^{(i)}\xleftarrow{\beta_i}\underline{Y}_{(i)}$ are as follows:
\begin{equation}
\xymatrix{
\underline{Y}_{(k)}\ar[d]_{\beta_k}:\ \ \bar{Y}_0\ar[r]&Y_1&\dots\ar[r]\ar[l]&Y_{k-1}\ar[d]^{\id}&\bar{Y}_{k-1}\ar[l]\ar[r]^{a}\ar[d]^{\id}&Y_k\ar[d]^{\id}&\bar{Y}_k\ar[l]_{b}\ar[r]^{\id}\ar[d]^{b}&\bar{Y}_k\ar[d]^{b}&\dots\ar[l]\\
\underline{Y}^{(k)}:\ \ \bar{Y}_0\ar[r]&Y_1&\dots\ar[r]\ar[l]&Y_{k-1}&\bar{Y}_{k-1}\ar[l]\ar[r]^{a}&Y_k&Y_k\ar[l]_{\id}\ar[r]^{\id}&Y_k&\dots\ar[l]\\
\underline{Y}_{(k-1)}\ar[u]^{\alpha_k}:\ \ \bar{Y}_0\ar[r]&Y_1&\dots\ar[r]\ar[l]&Y_{k-1}\ar[u]_{\id}&\bar{Y}_{k-1}\ar[l]\ar[r]^{\id}\ar[u]_{\id}&\bar{Y}_{k-1}\ar[u]_{a}&\bar{Y}_{k-1}\ar[l]_{\id}\ar[r]^{\id}\ar[u]_{a}&\bar{Y}_{k-1}\ar[u]_{a}&\dots\ar[l]
}
\end{equation}
We are done.
\end{proof}

\begin{coroll}\label{corre}
The categories $\Lambda_n\mathscr{D}$, $n\ge 0$, are homotopy equivalent to the category $\mathscr{D}$.
\end{coroll}
\qed

(Note that there is a dual construction $\Lambda^*(\mathscr{D})$ obtained by reversing of the oriantation of all arrows, and two maps $q_0,q_1\colon \Lambda^*(\mathscr{D})\to\mathscr{D}$, for which the analogs of Proposition \ref{prope2} and Corollary \ref{corre} hold. We leave the detail the details to the interested reader.)

We want to deduce that $\Lambda_*\mathscr{D}$ is as well homotopy equivalent to $\mathscr{D}$. This point requires some accuracy, as what we get is the following.

For each $n\ge 0$, $k\le n$, and $\underline{Y}\in\Lambda_n\mathscr{D}$, the functors $P^{n,k}\colon \underline{Y}\mapsto
 \underline{Y}^{(k)}$ and $P_{n,k}\colon \underline{Y}\to\underline{Y}_{(k)}$ are compatible with the action of $\Delta_\str$ (correspondingly, $\Delta_\le$), and, by Lemma \ref{ntgr}(2), give rise to functors 
$$
P^k,P_k\colon \Lambda_*\mathscr{D}\to \Lambda_*\mathscr{D}
$$
Now maps of functors $\alpha_k,\beta_k$ give rise, by Lemma \ref{ntgrbis}, to maps of functors
\begin{equation}\label{zz1}
P_0\xleftarrow{[\alpha_1]_c} P^1 \xrightarrow{[\beta_1]_c} P_1\xleftarrow{[\alpha_2]_c}  P^2 \xrightarrow{[\beta_2]_c} P_2\leftarrow\dots
\end{equation}
such that, for any fixed $\Omega=([n],\underline{Y})\in \Lambda_*\mathscr{D}$ there exists $N(n)$ such that 
$$
P_{s}(\Omega)=\Omega\text{  and  }P^s(\Omega)=\Omega
$$
for any $s\ge N(n)=N(n(\Omega))$. (Whence $P_0(n,\underline{Y})=\underline{Y}_{(0)}$ independently on $n$).
To conclude that the map $(n,\underline{Y})\mapsto \underline{Y}_{(0)}$ is homotopic to the identity map of $\Lambda_*\mathscr{D}$, we recall some results of [Mi] in the next Subsection.

\subsection{\sc Different homotopy relations}\label{diffhom}
We have at least two different concepts of ``homotopy of functors'' between small categories.

Let $\mathscr{C}_1,\mathscr{C}_2$ be two small categories, $f,g\colon \mathscr{C}_1\to\mathscr{C}_2$ two functors.

We write $f\simeq g$ if there is a zig-zag of maps of functors
$$
f_0\rightarrow f_1\leftarrow f_2\rightarrow\dots \rightarrow f_n
$$
with $f_0=f$, $f_N=g$. 
In that case we say that $f$ and $g$ are {\it homotopic} functors.

We say that $f$ and $g$ are {\it $B$-homotopic} functors if the maps $Bf,Bg\colon B\mathscr{C}_1\to B\mathscr{C}_2$ are homotopic maps of topological spaces. In that case we write $f\simeq_B g$.

It follows from Proposition \ref{propbasic}(i) that $f\simeq g$ implies $f\simeq_B g$.

It turns out that there is another homotopy relation between functors, denoted $f\simeq_E g$ and called {\it Evrard homotopy relation} such that one has
$$
f\simeq g\Rightarrow f\simeq_E g\Rightarrow f\simeq_B g
$$
and both implications are strict [Mi, Remark 2.9].

We follow [Mi] (where the author uses the notation $f\simeq_H g$ for our $f\simeq_E g$, and call it {\it Hoff homotopy} relation).

Denote by $I_n$ the category with objects $0,1,\dots,n$ defined by the graph
$$
0\rightarrow 1\leftarrow 2\rightarrow 3\leftarrow\dots \rightarrow n
$$
(for $n$ odd, when $n$ is even the rightmost arrow is directed leftward). 

Let $\mathscr{N}$ be the category with objects $0,1,2,\dots$ given by the graph
$$
0\rightarrow 1\leftarrow 2\rightarrow 3\leftarrow \dots
$$
Alternatively, one can define
$$
\mathscr{N}=\Delta_\le\int_c I_*
$$
\begin{defn}\label{deffin}{\rm
A functor $f\colon\mathscr{N}\to\mathscr{C}$ is {\it finite} if there exists $m$ such that $f(m)=f(n)$ for all $n\ge m$, and the corresponding morphisms in $\mathscr{N}$ are mapped to $\id_{f(m)}$.
}
\end{defn}

We denote by $\mathscr{N}(\mathscr{C})$ the category of all finite functors $\mathscr{N}\to\mathscr{C}$.

For any functor $h\colon \mathscr{N}\to \mathscr{C}$ we denote by $\alpha(h)\in\mathscr{C}$ the composition $\{0\}\to\mathscr{N}\xrightarrow{h}\mathscr{C}$.

For any {\it finite} functor $h\in\mathscr{N}(\mathscr{C})$ the ``end-value'' $\omega(h)\in\mathscr{C}$ is defined as 
$h(m)\in\mathscr{C}$ where $m$ is the number in the definition of a finite functor.

\begin{defn}{\rm
Two functors $f,g \colon\mathscr{C}\to\mathscr{D}$ are called {\it Evrard homotopic} if there is a functor 
$H\colon \mathscr{C}\to\mathscr{N}(\mathscr{D})$ such that $\alpha(H)\colon \mathscr{C}\to\mathscr{D}$ is equal to $f$, and $\omega(H)\colon\mathscr{C}\to\mathscr{D}$ is equal to $g$.
}
\end{defn}

It is clear that $f\simeq g\Rightarrow f\simeq_E g$. This implication is strict [Mi, Remark 2.9]. 
The matter is that for each object $c\in\mathscr{C}$ the ``length'' (the minimal number $m$ in Definition \ref{deffin} applied to $H(c)$ is a function on $\mathscr{C}$. If $\mathscr{C}$ has infinitely many objects, this function may be not bounded, considered as a function $\Ob\mathscr{C}\to\mathbb{N}$.

The main result [Mi, Remark 2.11] which we use here is the following:
\begin{prop}\label{mitheorem}
For functors $f,g\colon\mathscr{C}\to\mathscr{D}$ between small categories there are the following (strict) implications:
$$
f\simeq g\Rightarrow f\simeq_E g \Rightarrow f\Rightarrow_B g
$$
\end{prop}

\qed

The link with our problem is the following.

For any fixed $n$, Proposition \ref{prope2} gives $\underline{p}_0\simeq \id_{\Lambda_n(\mathscr{D})}$, in the sense of the strongest homotopy relation $\simeq$. Then our discussion below Corrolary \ref{corre} shows that the same construction as in Proposition \ref{prope2} gives a functor
$$
H\colon \Lambda_*(\mathscr{D})\to\mathscr{N}(\Lambda_*(\mathscr{D}))
$$
such that $\alpha(H)=\underline{p}_0$, $\omega(H)=\id_{\Lambda_*(\mathscr{D})}$.

That is, $\underline{p}_0\simeq_E \id_{\Lambda_*(\mathscr{D})}$, in the sense of the (weaker than $\simeq$) homotopy relation $\simeq_E$.

\begin{theorem}\label{itheorem1}
For any small category $\mathscr{D}$, the topological space $B(\Lambda_*(\mathscr{D}))$ is homotopy equivalent to the topological space $B(\mathscr{D})$. This homotopy equivalence is induced by the functor $i\colon \mathscr{D}\to\Lambda_*(\mathscr{D})$ assigning to $Y\in \mathscr{D}$ the object $([0],Y)$ of length 0. Its homotopy inverse is induced by the functor assigning to $([n],\underline{Y})$ its 0-component $Y_0$.
\end{theorem}
\begin{proof}
We define embedding $i\colon \mathscr{D}\to\Lambda_*(\mathscr{D})$ sending $Y$ to $([0],Y)$. Denote by $p\colon \Lambda_*(\mathscr{D})\to \mathscr{D}$ the functor assigning to $([n],\underline{Y})$ its 0-part $Y_0$. 
Our previous discussion shows that $\id_{\Lambda_*(\mathscr{D})}\simeq_E (i\circ p)$. It implies by Proposition \ref{mitheorem} that
$B(\id), B(i)\circ B(p)\colon B(\Lambda_*(\mathscr{D}))\to B(\Lambda_*(\mathscr{D}))$ are homotopic maps of topological spaces. 
But $p\circ i=\id_{\mathscr{D}}$, what gives another (identical) homotopy.
\end{proof}

\subsection{\sc The Evrard homotopy fibrant replacement of a functor}
For a functor $f\colon\mathscr{C}\to\mathscr{D}$, define the category $\mathscr{H}_*(f)$ by the following pullback diagram:
\begin{equation}
\xymatrix{
\mathscr{H}_*(f)\ar[r]^{\hat{f}}\ar[d]&\Lambda_*\mathscr{D}\ar[d]^{p_0}\\
\mathscr{C}\ar[r]^{f}&\mathscr{D}
}
\end{equation}
(Recall that the subscript $*$ means either $\str$ or $\le$, see Section \ref{sectionpath}).
The functor $p_1\colon\Lambda_*\mathscr{D}\to\mathscr{D}$ defines the composition $\mathscr{H}_*(f)\to\Lambda_*\mathscr{D}\xrightarrow{p_1}\mathscr{D}$ which is denoted by $f_h$.

Explicitly, an object of $\mathscr{H}_*(f)$ is a triple $(X,[n],\underline{Y})$ with $X\in\mathscr{C}$, $\underline{Y}\in\Lambda_n\mathscr{D}$, and with $\bar{Y}_0=f(X)$. Then $p_1$ assigns to $(X,[n],\underline{Y})$ the rightmost object $\bar{Y}_n$  in the string $\underline{Y}$.

There are functors $q:\mathscr{H}_*(f)\to\mathscr{C}$, which assigns to $(X,\underline{Y})$ the object $X$, and $i\colon \mathscr{C}\to\mathscr{H}(f)$, which assigns to $X$ the pair $(X,\underline{f(X)})$, where $\underline{f(X)}$ is of length 0 whose only component is $f(X)$.

(There is a dual construction defined as the composition $\mathscr{H}^*(f)\to \Lambda^*(\mathscr{D})\xrightarrow{q_0}\mathscr{D}$, which is denoted by $f^h$).

\begin{prop}\label{prope3}
Let $f\colon\mathscr{C}\to\mathscr{D}$ be a functor. Then
the functors $i$ and $q$ induce homotopy equivalences $B(\mathscr{C})\rightleftarrows B(\mathscr{H}_*(f))$, homotopy inverse to each other. Therefore, the categories $\mathscr{C}$ and $\mathscr{H}(f)$ are $\simeq_B$-homotopy equivalent, see Section \ref{diffhom}.
\end{prop}
\begin{proof}
We have $q\circ i=\id$, and the composition $i\circ q$ is the functor assigning to $(X,[n],\underline{Y})$ the object $(X,[0],f(X))$.
We prove that it is Evrard homotopic to the identity functor of $\mathscr{H}_*(f)$, following the same line as in Proposition \ref{prope2}. Then we apply Proposition \ref{mitheorem}, as in the proof of Theorem \ref{itheorem1}. 
\end{proof}
Consider the functor $f_h\colon\mathscr{H}_*(f)\to \mathscr{D}$ defined above. 

The case $*=\str$ of the following Theorem is partially contained in the Thesis of Evrard [Ev1]. Our proof works simultaneously for both cases and yields a more precise homotopical information than Evrard's treatment [Ev1, Ev2].
\begin{theorem}\label{theoreme}
Any functor $f\colon \mathscr{C}\to\mathscr{D}$ admits a (natural) factorization into a homotopy equivalence $i_f\colon \mathscr{C}\to\mathscr{H}_*(f)$ followed by an $h$-fibred functor $f_h\colon \mathscr{H}_*(f)\to\mathscr{D}$. 

Dually, the functor $f$ also admits a (natural) factorization into a homotopy equivalence $j_f\colon \mathscr{C}\to\mathscr{H}^*(f)$ followed by an $h$-cofibred functor $f^h\colon \mathscr{H}^*(f)\to\mathscr{D}$.

In particular, the functor $f_h$ (resp., $f^h$) fulfills the assumption (resp. the dual assumption) of Quillen Theorem B (cf. Theorem \ref{qtheoremb} and Proposition \ref{propm}).
\end{theorem}
We prove Theorem \ref{theoreme} in the next Section.

\begin{remark}{\rm
In the extended electronically available version of [Ma], Maltsiniotis proves the existence of factorization which is different but related to ours, see [Ma, Theorem 3.2.45].
}
\end{remark}
\section{\sc Proof of Theorem \ref{theoreme}}
We provide a proof for the case of $f_h$, the interested reader is invited to carry out the dual part of Theorem \ref{theoreme} just by dualizing the proofs involving $f_h$.

In Section \ref{sectioni1}, we prove that $f_h$ is a fibred functor.
Then in Section 3.2 we show that $f_h$ is $h$-fibred, i.e. the induced base-change functors are weak homotopy equivalences.

We discuss in parallel the cases $*=\str$ and $*=\le$, though in a few place we omit some detail for $*=\str$ to avoid long routine but straightforward computations. 
\subsection{\sc $f_h$ is a fibred functor}\label{sectioni1}
Let $f\colon\mathscr{C}\to\mathscr{D}$ be a functor. Consider $f_h\colon\mathscr{H}_*(f)\to\mathscr{D}$.
\begin{lemma}
The functor $f_h$ is pre-fibred (and fibred). 
\end{lemma}
\begin{proof}
Recall the natural embedding $i_Y\colon f_h^{-1}(Y)\to Y\setminus f_h$, see Section \ref{sectionq}. To see that $f_h$ is pre-fibred, we need to show {\it that $i_Y$ admits a right adjoint}.

An element of the category $Y\setminus f_h$, $Y\in\mathscr{D}$ is a tuple
$$
\omega=(X,\underline{Y}_n, \sigma) \text{  where  }X\in\mathscr{C},\underline{Y}_n\in \Lambda_n(\mathscr{D}), \bar{Y}_0=f(X), \sigma\colon Y\to \bar{Y}_n=f_h(\underline{Y}_n)
$$
We use notations as in \eqref{notpath} for an element in $\Lambda_n(\mathscr{D})$, that is, $\underline{Y}_n$ is of the form
\begin{equation}\label{ww}
\bar{Y}_0\rightarrow Y_1\leftarrow \bar{Y}_1\rightarrow Y_2\leftarrow\bar{Y}_2\rightarrow Y_3\leftarrow \dots\rightarrow Y_n\xleftarrow{a}\bar{Y}_n
\end{equation}
Define the right adjoint $R_Y$ to $i_Y$ on the object $\omega$ as $R_Y(\omega)=(X,\underline{Z}_n)$ where $\underline{Z}_n\in\Lambda_n(\mathscr{D})$ is the following string:
\begin{equation}
\bar{Y}_0\rightarrow Y_1\leftarrow \bar{Y}_1\rightarrow Y_2\leftarrow\bar{Y}_2\rightarrow Y_3\leftarrow \dots\rightarrow Y_n\xleftarrow{a\circ \sigma}{Y}
\end{equation}
To define $R_Y$ on morphisms, let $\Phi\colon \omega\to\omega^\prime$ be a morphism in $Y\setminus f_h$.
It means that we are given another object
$$
\omega^\prime=(X^\prime\in\mathscr{C},\underline{Y}_m^\prime\in\Lambda_m(\mathscr{D}),
\bar{Y}^\prime_0=f(X^\prime), \sigma^\prime\colon Y\to \bar{Y}^\prime_m)
$$
in $Y\setminus f_h$, and morphisms
$$
s\colon X\to X^\prime, w\colon \Lambda(\phi)(\omega)\to\omega^\prime
$$
such that $w_0=f(s)$ and 
\begin{equation}\label{eqw0}
\bar{w}_m\circ \sigma=\sigma^\prime
\end{equation}
Note that in general $m\ne n$, $\phi\colon [n]\to [m]$ is a morphism in $\Delta_\str$ for $*=\str$ and in $\Delta_\le$ for $*=\le$; one always has $m\ge n$. Here $w\colon \Lambda(\phi)(\omega)\to\omega^\prime$ is a morphism in $\Lambda_m(\mathscr{D})$. 

The corresponding morphism $R_Y(s,w)\colon R_Y(\omega)\to R_Y(\omega^\prime)$ is defined as 
$$
(s\colon X\to X^\prime, R_Y(w)\colon \Lambda(\phi)(R_Y(\omega))\to R_Y(\omega^\prime))
$$
To simplify the notations, we construct $R_Y(w)$ only for the case $*=\le$; the case $*=\str$ is similar and is left to the reader.
So we let $*=\le$, $m-n=k\ge 0$, there is only one morphism $n\to m$ in $\Delta_{\le}$. We use the following notations for the morphism $w\colon \Lambda(\phi)(\omega)\to\omega^\prime$:
\begin{equation}\label{eqw1}
\xymatrix{
\dots\bar{Y}_{n-1}\ar[r]\ar[d]^{\bar{w}_{n-1}}   &Y_{n}\ar[d]^{w_m} &\bar{Y}_n\ar[l]_{a}\ar[d]^{\bar{w}_m}\\
\dots\bar{Y}_{m-1}^\prime\ar[r]&Y^\prime_m&\bar{Y}_m\ar[l]_{a^\prime}
}
\end{equation}
Define $R_Y(w)\colon \Lambda(\phi)(R_Y(\omega))\to R_Y(\omega^\prime))$ as
\begin{equation}\label{eqw2}
\xymatrix{
\dots\bar{Y}_{n-1}\ar[r]\ar[d]^{\bar{w}_{m-1}} &Y_{n}\ar[d]^{w_m} &Y\ar[l]_{a\circ \sigma} \ar[d]^{\id}\\
\dots\bar{Y}_{m-1}\ar[r]&Y_m^\prime&Y\ar[l]_{a^\prime\circ \sigma^\prime}
}
\end{equation}
To see that the definition is correct, we only need to check that the rightmost square in \eqref{eqw2} is commutative. It follows from the commutativity of the rightmost square in \eqref{eqw1} and from \eqref{eqw0}.

Now check that the functor $R_Y$ is right adjoint to $i_Y$ (again in the case $*=\le$).

Let $\omega\in Y\setminus f_h$ as above, and let $\eta=(X,\underline{Z}_\ell)$, $Z_0=f(X),\bar{Z}_\ell=Y$, be an element in $f_h^{-1}(Y)$. We need to bi-functorially identify
\begin{equation}\label{eqw3}
\Hom_{f_h^{-1}(Y)}(\eta,R_Y(\omega))\simeq\Hom_{Y\setminus f_h}(i_Y(\eta),\omega)
\end{equation}
We have $n\ge \ell$, otherwise both sides of \eqref{eqw3} are empty sets. We track the rightmost square in both sides. They are
\begin{equation}\label{eqw4}
\xymatrix{
\dots Z_\ell\ar[d]^{w_n}&Y\ar[l]_{b}\ar[d]^{\id}\\
\dots Y_n&Y\ar[l]_{a\circ \sigma}
}
\end{equation}
for the l.h.s. of \eqref{eqw3}, and
\begin{equation}\label{eqw5}
\xymatrix{
\dots Z_\ell\ar[d]^{w_n}&Y\ar[l]_{b}\ar[d]^{\bar{w}_n}\\
\dots Y_n&\bar{Y}_n\ar[l]_{a}
}
\end{equation}
with 
\begin{equation}\label{eqw6}
\bar{w}_n=\bar{w}_n\circ \id_Y=\sigma
\end{equation}
for the r.h.s. of \eqref{eqw3}.

It is clear that there is 1-to-1 correspondence between diagrams \eqref{eqw4} and diagrams \eqref{eqw5} with $\sigma=\bar{w}_n$.

It proves that the functor $f_h$ is pre-fibred. In fact, it is fibred (though we don't use it), see Section \ref{sectioni2} below.
\end{proof}

\subsection{\sc The base-change functor}\label{sectioni2}
With any morphism $v\colon Y\to Y^\prime$ in $\mathscr{D}$ is associated the base-change morphism 
$$
v^*=R_Y\circ [v^*]\circ i_Y\colon f_h^{-1}(Y^\prime)\to f_h^{-1}(Y)
$$
see \eqref{bc}.

In our case, the base-change morphism is just the pre-composition of the rightmost map $a\colon Y^\prime=\bar{Y}_n\to Y_n$ in in the string \eqref{ww} with  $v\colon Y\to Y^{\prime}$. Thus, it sends $\omega=(X,\underline{Y}_n)$ with $\underline{Y}_n$ as
$$
\bar{Y}_0\rightarrow Y_1\leftarrow \bar{Y}_1\rightarrow Y_2\leftarrow\bar{Y}_2\rightarrow Y_3\leftarrow \dots\rightarrow Y_n\xleftarrow{a}Y^\prime
$$
to
$$
\bar{Y}_0\rightarrow Y_1\leftarrow \bar{Y}_1\rightarrow Y_2\leftarrow\bar{Y}_2\rightarrow Y_3\leftarrow \dots\rightarrow Y_n\xleftarrow{a\circ v}Y
$$
and is extended to morphisms in a natural way.

Note that for two composable morphisms $v_1$ and $v_2$ in $\mathscr{D}$, one has 
\begin{equation}\label{wf}
(v_2\circ v_1)^*=v_1^*\circ v_2^*
\end{equation}
what can be seen from the description of $v^*$ given above. In general, the two sides of \eqref{wf} are only isomorphic; a pre-fibred functor for which one can choose representatives for all right adjoints $R_Y$ (any two representatives are isomorphic) such that \eqref{wf} holds on the nose (not just up to an isomorphism) is called {\it fibred}.

We have:
\begin{lemma}\label{theoremee}
In the notations as above, any base-change functor
$v^*\colon f_h^{-1}(Y^\prime)\to f_h^{-1}(Y)$ is a weak homotopy equivalence of categories.
\end{lemma}

{\it Proof of the Lemma:}

With a morphism $v\colon Y\to Y^\prime$, we associated the base-change functor $v^*\colon f_h^{-1}(Y^\prime)\to f_h^{-1}(Y)$.
Note that $v^*$ is well-defined as well on the strings $\underline{Y}_n$ of fixed length $n$. Our strategy is as follows:
\begin{itemize}
\item[(i)] We construct a functor $v_\dagger\colon f_h^{-1}(Y)\to f_h^{-1}(Y^\prime)$, associated with the same morphism $v\colon Y\to Y^\prime$ in $\mathscr{D}$. It increases the length of string $\underline{Y}_n$ by 1, so it can not be specified for fixed $n$, but it is defined as a functor between the corresponding Grothendieck constructions;
\item[(ii)] We construct a shift functor $T\colon f_h^{-1}(Y)\to f_h^{-1}(Y)$ increasing the length of a string $\underline{Y}_n$ by 1. We prove that $T$ is weak homotopy equivalent to the identity functor, on the level of Grothendieck constructions;
\item[(iii)] We construct natural transformations 
$$
A\colon T\to v^*v_\dagger
$$
and
$$
B\colon v_\dagger v^*\to T
$$
\end{itemize}
Then Theorem \ref{theoremee} will follow, by Proposition \ref{propbasic}.

\subsubsection{\sc The shift functor}
Define the functor $T_n\colon \Lambda_n(\mathscr{D})\to\Lambda_{n+1}(\mathscr{D})$ as
\begin{equation}
T_n(\underline{Y})=\dots\leftarrow Y_{n-1}\rightarrow\bar{Y}_{n-1}\leftarrow Y_n\rightarrow\bar{Y}_n\xleftarrow{\id}\bar{Y}_n\xrightarrow{\id}\bar{Y}_n
\end{equation}
where $\underline{Y}\in\Lambda_n(\mathscr{D})$.

The functors $T_n$ define a functor $T\colon \Lambda_*(\mathscr{D})\to \Lambda_*(\mathscr{D})$, for both cases $*=\str$ and $*=\le$, by Lemma \ref{ntgr}(2).
For a functor $f\colon \mathscr{C}\to\mathscr{D}$, the functors $T_n$ define a functor
$T\colon \mathscr{H}_*(f)\to\mathscr{H}_*(f)$, by
\begin{equation}
T(X,[n],\underline{Y})=(X,[n+1],T_n(\underline{Y}))
\end{equation}
The functor $T$ preserves the category $f_h^{-1}(Y)$ for any $Y\in\mathscr{D}$. We denote the restriction of $T$ to $f_h^{-1}(Y)$ by $T_Y$.
One has
\begin{prop}\label{lemmae1}
Let $f\colon\mathscr{C}\to\mathscr{D}$ be a functor. Then $T\colon\mathscr{H}_*(f)\to\mathscr{H}_*(f)$ is a weak homotopy equivalence, and is homotopy equivalent to the identity functor.  The corresponding functor
$T_Y\colon f_h^{-1}(Y)\to f_h^{-1}(Y)$ as well is a homotopy equivalence, homotopy equivalent to the identity functor.
\end{prop}
\begin{proof}
Let $\underline{Y}_n\in\Lambda_n(\mathscr{D})$, consider $T(\underline{Y}_n)\in\Lambda_{n+1}(\mathscr{D})$.
Denote by $j\colon n\to n+1$ the morphism in $\Delta_{\le}$. We construct functors
$$
T_0,T_1,\dots,T_{2n}\colon\Lambda_n(\mathscr{D})\to \Lambda_{n+1}(\mathscr{D})
$$
with
$$
T_0=\Lambda(j),\ \ T_{2n}=T
$$
and a zig-zag of maps of functors
$$
T_0\xrightarrow{\theta_0}T_1\xleftarrow{\theta_1}T_2\xrightarrow{\theta_2}T_3\dots   T_{2n-2} \xrightarrow{\theta_{2n-2}}T_{2n-1}\xleftarrow{\theta_{2n-1}} T_{2n}
$$
For even $i=2k$, $T_{2k}(\underline{Y}_n)\in\Lambda_{n+1}(\mathscr{D})$ is
$$
\bar{Y}_0\rightarrow Y_1\leftarrow \bar{Y}_1\rightarrow\dots       \rightarrow Y_k \leftarrow  \bar{Y}_{k} \xrightarrow{\id} \bar{Y}_{k}\xleftarrow{\id}  \bar{Y}_{k} \rightarrow Y_{k+1}\leftarrow \bar{Y}_{k+1}\rightarrow\dots\rightarrow Y_n\leftarrow \bar{Y}_n
$$
(that is, two extra $\bar{Y}_k$ are inserted to the string, and the new arrows are the identity maps). 

For odd $i=2k+1$, $T_{2k+1}(\underline{Y}_n)\in\Lambda_{n+1}(\mathscr{D})$ is
$$
\bar{Y}_0\rightarrow Y_1\leftarrow \bar{Y}_1\rightarrow\dots  \leftarrow \bar{Y}_k\rightarrow Y_{k+1}\xleftarrow{\id}Y_{k+1}\xrightarrow{\id}Y_{k+1}\leftarrow\bar{Y}_{k+1}\rightarrow Y_{k+2}\leftarrow\dots\rightarrow Y_n\leftarrow \bar{Y}_n
$$
(that is, two extra ${Y}_{k+1}$ are inserted to the string, and the two new arrows are the identity maps). 

We construct the two ``rightmost'' maps of functors, $\theta_{2n-1}$ and $\theta_{2n-2}$, the remaining maps are similar. These two rightmost maps are:
\begin{equation}
\xymatrix{
T_{2n}\ar[d]_{\theta_{2n-1}}&&\dots\ar[r] & Y_{n-1}\ar[d]^{\id}&\bar{Y}_{n-1}\ar[r]^{b}\ar[l]_{c}\ar[d]^{\id}&Y_n\ar[d]^{\id}&\bar{Y}_n\ar[r]^{\id}\ar[l]_{a}\ar[d]^{a}  &\bar{Y}_n\ar[d]^{a} & \bar{Y}_n\ar[l]_{\id}\ar[d]^{\id}\\
T_{2n-1}&&\dots\ar[r]&Y_{n-1}&\bar{Y}_{n-1}\ar[l]_{c}\ar[r]^{b}&Y_n&Y_n\ar[l]_{\id}\ar[r]^{\id}&Y_n&\bar{Y}_n\ar[l]_{a}\\
T_{2n-2}\ar[u]^{\theta_{2n-2}}&&\dots\ar[r]&Y_{n-1}\ar[u]_{\id}&\bar{Y}_{n-1}\ar[l]_{c}\ar[r]^{\id}\ar[u]_{\id}&
\bar{Y}_{n-1}\ar[u]_{b}&\bar{Y}_{n-1}\ar[u]_{b}\ar[l]_{\id}\ar[r]^{b}&Y_n\ar[u]_{\id}&\bar{Y}_n\ar[l]_{a}\ar[u]_{\id}
}
\end{equation}
The same formulas define a zig-zag of homotopy equivalences  for $\mathscr{H}_*(f)$.

We also have a morphism $\Lambda(j)(\underline{Y})\to\underline{Y}$ in the cofibrant Grothendieck construction, which gives rise to a map of functors $T_0\to\id$ in $\Delta_*\int_c\Lambda$, and to a map of functors $T_0\to\id$ in $\mathscr{H}_*(f)$. 

That is, we have the following zig-zag of maps of functors in $\mathscr{H}_*(f)$:
$$
\id\leftarrow T_0\xrightarrow{\theta_0}T_1\xleftarrow{\theta_1}T_2\xrightarrow{\theta_2}T_3\dots   T_{2n-2} \xrightarrow{\theta_{2n-2}}T_{2n-1}\xleftarrow{\theta_{2n-1}} T_{2n}
$$
Now the argument analogous to the one in Section \ref{diffhom} shows that $T$ is $\simeq_E$-homotopic to the identity functor. It implies, by Proposition \ref{mitheorem}, that $T$ and $\id$ define homotopic maps on $B(\mathscr{H}_*(f))$.

\comment
We construct a natural transformation $\id_{\mathscr{H}(f)}\to T_f$, then the result follows from Proposition \ref{propbasic}(i).
That is, we need to define for any $(X,[n],\underline{Y})$ a morphism in $\mathscr{H}(f)$
$$
\theta:(X,[n],\underline{Y})\to (X,[n+1], T_n(\underline{Y}))
$$
such that for any morphism $g$ in $\mathscr{H}(f)$ the diagram
\begin{equation}\label{e1}
\xymatrix{
(X,[n],\underline{Y})\ar[r]^{\theta}\ar[d]_{g}&(X,[n+1],T_n(\underline{Y}))\ar[d]^{g}\\
(X^\prime,[m],\underline{Y}^\prime)\ar[r]^{\theta}&(X^\prime,[m+1],T_m(\underline{Y}^\prime))
}
\end{equation}
commutes.

We define the component $\theta([n])$ as the map $\theta_n: [n]\to [n+1]$ with $\theta_n(i)=i$, $1\le i\le n$. Then the corresponding map
$\Lambda(\theta_n)(\underline{Y})\to T_n(\underline{Y})$ is defined as
the identity map (note that $\Lambda(\theta_n)(\underline{Y})=T_n(\underline{Y})$).
The commutativity of \eqref{e1} for any $g$ is clear.
\endcomment
\end{proof}

\subsubsection{\sc The functor $v_\dagger$}
Let $v\colon Y\to Y^\prime$ be a morphism in $\mathscr{D}$. 
We construct a functor 
$$
v_\dagger\colon f_h^{-1}(Y)\to f_h^{-1}(Y^\prime)
$$
on an object
$$
\dots\rightarrow Z_{n-1}\leftarrow\bar{Z}_{n-1}\rightarrow Z_n\xleftarrow{b}\underset{=Y}{\bar{Z}_n}
$$
as
$$
\dots\rightarrow Z_{n-1}\leftarrow\bar{Z}_{n-1}\rightarrow Z_n\xleftarrow{b}\underset{=Y}{\bar{Z}_n}\xrightarrow{v}Y^\prime\xleftarrow{\id}Y^\prime
$$
The Grothendieck construction gives then a functor $v_\dagger$.

\begin{prop}\label{propx}
Let $v\colon Y\to Y^\prime$ be a morphism in $\mathscr{D}$. There are maps of functors
$$
A\colon T\Rightarrow v^*v_\dagger\colon f_h^{-1}(Y)\to f_h^{-1}(Y)
$$
and
$$
B\colon v_\dagger v^*\Rightarrow T\colon f_h^{-1}(Y^\prime)\to f_h^{-1}(Y^\prime)
$$
\end{prop}
\begin{proof}
Denote by $f_h^{-1}(Y)[n]$ be the category of the corresponding strings of length $n$, thus 
$$
f_h(Y)=\Delta_*\int_c f_h^{-1}(Y)[n]
$$
We construct maps of functors
$$
A[n]\colon T\Rightarrow v^*v_\dagger\colon f_h^{-1}(Y)[n]\to f_h^{-1}(Y)[n+1]
$$
and
$$
B[n]\colon v_\dagger v^*\Rightarrow T\colon f_h^{-1}(Y^\prime)[n]\to f_h^{-1}(Y^\prime)[n+1]
$$
Then the result will follow from Lemma \ref{ntgrbis}.

For fixed $n$, $A[n]$ and $B[n]$ are given by the following diagrams, correspondingly:

\begin{equation}
\xymatrix{
\dots\ar[r]&Z_n&{\overset{=Y}{\bar{Z}_n}}\ar[l]_{b}\ar[r]^{v}&Y^\prime&Y\ar[l]_{v}\\
\dots\ar[r]&Z_n\ar[u]^{\id}&{\underset{=Y}{\bar{Z}_n}}\ar[l]_{b}\ar[u]^{\id}\ar[r]^{\id}&Y\ar[u]^{v}&Y\ar[l]_{\id}\ar[u]^{\id}
}
\end{equation}
\begin{equation}
\xymatrix{
\dots\ar[r]&Y_n\ar[d]_{\id}&{\overset{=Y^\prime}{\bar{Y}_n}}\ar[l]_{av}\ar[d]_{v}\ar[r]^{v}&Y^{\prime}\ar[d]_{\id}&Y^{\prime}\ar[l]_{\id}\ar[d]_{\id}\\
\dots\ar[r]&Y_n&{\underset{=Y^\prime}{\bar{Y}_n}}\ar[l]_{a}\ar[r]^{\id}&Y^\prime&Y^\prime\ar[l]_{\id}
}
\end{equation}

One easily checks that both $A[n]$ and $B[n]$ form a 3-arrow, when $n$ runs through the objects of the category $\Delta_*$ (see the definition just before Lemma \ref{ntgrbis}), thus Lemma \ref{ntgrbis} applies.

We are done.
\end{proof}

Lemma \ref{theoremee} is proven.

\qed

Now Theorem \ref{theoreme} follows from Propositions \ref{mitheorem}, \ref{lemmae1}, and \ref{propx}.

\qed

\comment
The functor $u_\dagger$ assigns to $\underline{Y}\in \Lambda_n(\mathscr{D})$ of the form
\begin{equation}\label{ee1}
 \dots \rightarrow \bar{Y}_{n-1}    \leftarrow     Y_n \xrightarrow{\alpha}  \underset{=Y}{\bar{Y}_n}
\end{equation}
the element
\begin{equation}
 \dots \rightarrow \bar{Y}_{n-1}    \leftarrow     Y_n \xrightarrow{u\circ\alpha}  Y^\prime
\end{equation}
in $\Lambda_n(\mathscr{D})$. It defines naturally a functor $u_\dagger\colon f_h^{-1}(Y)\to f_h^{-1}(Y^\prime)$.
The functor $u^\dagger$ assigns to
$\underline{Z}\in \Lambda_n(\mathscr{D})$ of the form
\begin{equation}\label{ee2}
 \dots \rightarrow \bar{Z}_{n-1}    \leftarrow     Z_n \xrightarrow{\beta}  \underset{=Y^\prime}{\bar{Z}_n}
\end{equation}
the following element in $\Lambda_{n+1}(\mathscr{D})$:
\begin{equation}
 \dots \rightarrow \bar{Z}_{n-1}    \leftarrow     Z_n \xrightarrow{\beta}  \underset{=Y^\prime}{\bar{Z}_n}\xleftarrow{u}Y\xrightarrow{\id}Y
\end{equation}
It defines naturally a functor $u^\dagger\colon f_h^{-1}(Y^\prime)\to f_h^{-1}(Y)$.

\begin{lemma}\label{lemmae2}
There are natural transformations $\theta_1:\id_{f_h^{-1}(Y)}\to u^\dagger\circ u_\dagger$ and
$\theta_2:u_\dagger\circ u^\dagger\to T_f\circ \id_{f_h^{-1}(Y^\prime)}$.
\end{lemma}
\begin{proof}
For $\underline{Y}\in \Lambda_n(Y)$ and $\underline{Z}\in\Lambda_n(Y^\prime)$ as in \eqref{ee1}, \eqref{ee2}, define $\theta_1(X,[n],\underline{Y})$ and $\theta_2(X,[n],\underline{Z})$ as
follows. The value
$\theta_1(X,[n],\underline{Y})$ is a morphism in $f_h^{-1}(Y)$ from $(X,[n],\underline{Y})$ to $(X,[n+1],u^\dagger\circ u_\dagger(\underline{Y}))$.
Define it as the identity on the component $X$, as the map $\phi\colon [n]\to [n+1]$ in $\Delta_\str$ such that $\phi(i)=i$ for all$i=1\dots n$,
and the component $\Lambda(\phi)(\underline{Y})\to u^\dagger\circ u_\dagger (\underline{Y})$ is given as the following morphism in $\Lambda_{n+1}(Y)$:
\begin{equation}
\xymatrix{
\dots\ar[r]&\bar{Y}_{n-1}\ar[d]^{\id}&Y_n\ar[r]^{\alpha}\ar[l]\ar[d]^{\id}&Y\ar[d]^{u}&Y\ar[l]_{\id}\ar[r]^{\id}\ar[d]^{\id}&Y\ar[d]^{\id}\\
                               \dots\ar[r]       & \bar{Y}_{n-1}        &Y_n\ar[r]^{u\circ\alpha} \ar[l]       &Y^\prime       &Y\ar[r]^{\id}\ar[l]_{u} &Y
}
\end{equation}
The value $\theta_2(X,[n],\underline{Z})$ is a morphism in $f_h^{-1}(Y^\prime)$ from $(X,[n+1],u_\dagger\circ u^\dagger (\underline{Z}))$ to
$(X,[n+1],T_f\circ \id(\underline{Z}))$. It is defined as the identity on $X$, as the identity on $[n+1]$, and the corresponding morphism
$u_\dagger\circ u^\dagger (\underline{Z})\to T_f(\underline{Z})$ in $\Lambda_{n+1}(Y^\prime)$ is given as
\begin{equation}
\xymatrix{
 \dots\ar[r]&\bar{Z}_{n-1}\ar[d]^{\id}&Z_n\ar[r]^{\beta}\ar[l]\ar[d]^{\id}   &Y^{\prime}\ar[d]^{\id}&Y\ar[l]_{u}\ar[r]^{u}\ar[d]^{u}&Y^\prime\ar[d]^{\id}\\
    \dots\ar[r]     &\bar{Z}_{n-1}  &Z_n\ar[r]^{\beta}\ar[l]  &Y^\prime&Y^\prime\ar[r]^{\id}\ar[l]_{\id}    & Y^\prime
}
\end{equation}
In both cases, these definitions are compatible with the morphisms in $f_h^{-1}(Y)$ (corresp., in $f_h^{-1}(Y^\prime)$), and thus define natural transformations.
 \end{proof}

\begin{coroll}\label{corolle1}
For any $u\colon Y\to Y^\prime$ in $\mathscr{D}$, the functors $u_\dagger\colon f_h^{-1}(Y)\to f_h^{-1}(Y^\prime)$ and
$u^\dagger\colon f_h^{-1}(Y^\prime)\to f_h^{-1}(Y)$ are homotopy equivalences, homotopy mutually inverse to each other.
\end{coroll}
It follows from Lemma \ref{lemmae1} and Lemma \ref{lemmae2}.

\qed

\subsection{\sc }
Let $u\colon Y\to Y^\prime$ be a morphism in $\mathscr{D}$. Consider the following diagram:
\begin{equation}
\xymatrix{
f_h^{-1}(Y)\ar[r]^{i_Y}\ar[d]_{u_\dagger}&f_h\setminus Y\ar[d]^{u_*}\\
f_h^{-1}(Y^\prime)\ar[r]^{i_{Y^\prime}}&f_h\setminus Y^\prime
}
\end{equation}
One sees immediately that it commutes.

We know that the left-hand side vertical arrow $u_\dagger$ is a homotopy equivalence, see Corollary \ref{corolle1}.
Therefore, to prove that $u_*$ is a homotopy equivalence, it is enough to prove that the horisontal maps $i_Y$ and $i_{Y^\prime}$ also are.

For this end, define a functor $\ell_Y\colon f_h\setminus Y\to f_{h}^{-1}(Y)$ as follows.
It sends $(X,[n],\underline{Y},s)$ where $s\colon \bar{Y}_n\to Y$ a morphism in $\mathscr{D}$, to
$(X,[n+1],\underline{Z})$, where $\underline{Z}$ is given as
\begin{equation}
\dots\leftarrow Y_n\rightarrow\bar{Y}_{n-1}\leftarrow Y_n \xrightarrow{\alpha} \bar{Y}_n\xleftarrow{\id}\bar{Y}_n\xrightarrow{s}Y
\end{equation}
and it is defined on the morphisms accordingly.

\begin{lemma}
Let $Y\in\mathscr{D}$. Then the composition $\ell_Y\circ i_Y=T_f$, where $T_f$ is the functor homotopy equivalent to the identity, see Lemma \ref{lemmae1}, and there is a natural transformation $\omega:\id_{f_h\setminus Y}\to i_Y\circ \ell_Y$.
\end{lemma}
\begin{proof}
The first claim is straightforward. For the second, consider the commutative diagram
\begin{equation}
\xymatrix{
\dots &Y_n\ar[l]\ar[r]^{\alpha}\ar[dd]^{\id}&\bar{Y}_n\ar[dd]^{\id}&\bar{Y}_n\ar[l]_{\id}\ar[r]^{\id}\ar[dd]^{\id}&\bar{Y}_n\ar[dd]^{s}\ar[rd]^{s}\\
&&&&&Y\\
\dots&Y_n\ar[l]\ar[r]^{\alpha}&\bar{Y}_n&\bar{Y}_n\ar[l]_{\id}\ar[r]^s&Y\ar[ur]_{\id}
}
\end{equation}
\end{proof}

\begin{coroll}
The functors $i_Y$ and $\ell_Y$ are homotopy equivalences, homotopy inverse to each other.
\end{coroll}
\qed

Theorem \ref{theoreme} is proven.

\qed

\endcomment

\comment
\section{\sc A variation on Evrard's result}
Here we prove a version of Theorem \ref{theoreme}, using slightly different construction of the free paths category of a category $\mathscr{D}$, which we denote here $\Lambda^\prime\mathscr{D}$. The idea is that the category $\Delta_\str$ (see Section 2.1) contains too much morphisms.
The observation is that the morphisms we really used in the course of proof of Theorem \ref{theoreme} are those which form the following subcategory of the category $\Delta_\str$.

Define the category $\Delta_{\le}$ having the objects $\{[1],[2],[3],\dots\}$, and the only morphism from $[m]$ to $[n]$ if $m\le n$, and the empty set of morphisms otherwise.
We interpret this morphism as a morphism $f$ in $\Delta_\str$ from  $1<2<\dots<m$ to $1<2<\dots<n$ such that $f(i)=i$ ($m\le n$).
It gives an imbedding of categories $\Delta_\le\to\Delta_\str$.

Recall the categories $\Lambda_n\mathscr{D}$, defined for a small category $\mathscr{D}$ and for $n\ge 1$.
The composition of the functor $\Delta_\le\to\Delta_\str$ with the functor
$\Lambda\colon \Delta_\str\to\Cat$, sending $[n]$ to $\Lambda_n\mathscr{D}$, and a morphism $\phi\colon[m]\to [n]$ to the functor $\Lambda(\phi)$ as in \eqref{lphi}, gives a functor
$$
\Lambda^\prime\colon \Delta_\le\to\Cat
$$
We set
\begin{equation}
\Lambda^\prime\mathscr{D}=\Delta_\le\int\Lambda^\prime
\end{equation}
This category $\Lambda^\prime\mathscr{D}$ is another candidate for the ``free paths category'' of $\mathscr{D}$.
There are two functors $p_0,p_1\colon\Lambda^\prime \mathscr{D}\to\mathscr{D}$, the start point and the end point of the path functors, as in the case of $\Lambda\mathscr{D}$, see Section 2.1.

Then we define, for a functor $f\colon\mathscr{C}\to\mathscr{D}$, the category $\mathscr{H}^\prime(f)$ from the cartesian diagram:
\begin{equation}
\xymatrix{
\mathscr{H}^\prime(f)\ar[r]\ar[d]&\Lambda^\prime\mathscr{D}\ar[d]^{p_0}\\
\mathscr{C}\ar[r]^{f}&\mathscr{D}
}
\end{equation}
There is a functor $f^\prime_h\colon\mathscr{H}^\prime(f)\to\mathscr{D}$ given by $(X,\underline{Y})\mapsto p_1(Y)$.

We have a direct analogue of Proposition \ref{prope3}:
\begin{prop}\label{prope4}
The functor $q\colon\mathscr{H}^\prime(f)\to\mathscr{C}$, assigning $X\in \mathscr{C}$ to $(X,\underline{Y})$, is a homotopy equivalence.
\end{prop}
The proof is literally the same.

\qed

The main result in this Section is:
\begin{theorem}\label{theoremebis}
Let $f\colon \mathscr{C}\to\mathscr{D}$ be a functor between small categories, and let $f^\prime_h\colon \mathscr{H}^\prime(f)\to\mathscr{D}$ be the functor constructed above. Then the functor $f^\prime_h\colon\mathscr{H}^\prime(f)\to\mathscr{D}$ fulfills the assumptions of Quillen Theorem B.
\end{theorem}
\begin{proof}
The only thing we should mention in addition to the proof of Theorem \ref{theoreme}, is that our main ``tool'' in the proof of Theorem \ref{theoreme} is the shift functors $T_n\colon \Lambda_n\mathscr{D}\to \Lambda_{n+1}\mathscr{D}$, introduced in Section 3.1, and the corresponding functor $T_f\colon
\mathscr{H}(f)\to\mathscr{H}(f)$ constructed out of them. That is, we can replace the category $\Delta_\str$ by any subcategory (containing all its objects), where the the functors $T_n$ can be defined. It is clear that the category $\Delta_\le$ is the minimal among such categories. Then the proof can be repeated straightforwardly.
\end{proof}
\endcomment

\bigskip

{\small
\noindent {\sc Universiteit Antwerpen, Campus Middelheim, Wiskunde en Informatica, Gebouw G\\
Middelheimlaan 1, 2020 Antwerpen, Belgi\"{e}}}

\bigskip

\noindent{{\it e-mail}: {\tt Boris.Shoikhet@uantwerpen.be}}

\end{document}